[1994/12/01]
\documentclass{ijmart-mod}
\chardef\bslash=`\\ 

\usepackage{graphicx}
\usepackage[breaklinks=true]{hyperref}
\usepackage{mathtools}
\usepackage{caption}

\newtheorem{thm}{Theorem}[section]
\newtheorem{cor}[thm]{Corollary}
\newtheorem{lem}[thm]{Lemma}
\newtheorem{prop}[thm]{Proposition}

\theoremstyle{definition}

\newtheorem{rem}[thm]{Remark}

\theoremstyle{remark}

\newcommand{\eval}[2][\right]{\relax
  \ifx#1\right\relax \left.\fi#2#1\rvert}

\begin{document}
\title{Local extrema for hypercube sections}

\author[L. Pournin]{Lionel Pournin}
\address{Universit{\'e} Paris 13, Villetaneuse, France}
\email{lionel.pournin@univ-paris13.fr}

\begin{abstract}
Consider the hyperplanes at a fixed distance $t$ from the center of the hypercube $[0,1]^d$. Significant attention has been given to determining the hyperplanes $H$ among these such that the $(d-1)$-dimensional volume of $H\cap[0,1]^d$ is maximal or minimal. In the spirit of a question by Vitali Milman, the corresponding local problem is considered here when $H$ is orthogonal to a diagonal or a sub-diagonal of the hypercube. It is proven in particular that this volume is strictly locally maximal at the diagonals in all dimensions greater than $3$ within a range for $t$ that is asymptotic to $\sqrt{d}/\!\log d$. At lower order sub-diagonals, this volume is shown to be strictly locally maximal when $t$ is close to $0$ and not locally extremal when $t$ is large. This relies on a characterisation of local extremality at the diagonals and sub-diagonals that allows to solve the problem over the whole possible range for $t$ in any fixed, reasonably low dimension.
\end{abstract}
\maketitle

\section{Introduction}\label{LEH.sec.0}

Given a fixed non-negative number $t$, consider the hyperplanes of $\mathbb{R}^d$ whose distance to the center of the hypercube $[0,1]^d$ is equal to $t$. In other words, these hyperplanes are tangent to the sphere of radius $t$ whose center coincides with that of $[0,1]^d$. Significant attention has been devoted to identifying, among these hyperplanes, the ones whose intersection with $[0,1]^d$ has the largest or the smallest possible $(d-1)$-dimensional volume. When $t$ is equal to $0$ (and the hyperplanes contain the center of $[0,1]^d$), this was solved by Keith Ball \cite{Ball1986} who proved that the $(d-1)$\nobreakdash-dimensional volume of the intersection with the hypercube $[0,1]^d$ of a hyperplane $H$ through its center is maximal precisely when $H$ is orthogonal to an order $2$ sub-diagonal of that hypercube, thereby solving a problem posed by Douglas Hensley \cite{Hensley1979}. Here, an \emph{order $n$ sub-diagonal of $[0,1]^d$} means the line segment between the centers of two opposite $(d-n)$\nobreakdash-dimensional faces of $[0,1]^d$. The order $d$ sub-diagonals of the hypercube are also simply referred to as its diagonals. The result from \cite{Ball1986} relies on a formula for the $(d-1)$\nobreakdash-dimensional volume of the intersection of an arbitrary hyperplane with the hypercube $[0,1]^d$ that takes the form of an improper integral. That formula, given below as Theorem \ref{LEH.sec.1.thm.1}, goes back to George P{\'o}lya's work \cite{Polya1913} and is discussed, for instance in \cite{Berger2010,FrankRiede2012,KonigKoldobsky2011,Zong2006}.

The case when $t$ is positive (and less than $\sqrt{d}/2$, the circumradius of the hypercube) was later considered by Vitali Milman who asked \cite{Konig2021b,KonigKoldobsky2011} whether the minimal and the maximal $(d-1)$-dimensional volume of $H\cap[0,1]^d$ is always achieved when $H$ is orthogonal to a diagonal or a sub-diagonal of $[0,1]^d$---here and in the remainder of the introduction, $H$ denotes a generic hyperplane at distance $t$ from the center of $[0,1]^d$. This question was solved, in dimensions $2$ and $3$ by Hermann K{\"o}nig and Alexander Koldobsky \cite{KonigKoldobsky2011}. In higher dimensions, it was shown by James Moody, Corey Stone, David Zach, and Artem Zvavitch~\cite{MoodyStoneZachZvavitch2013} that, if $t$ is greater than $\sqrt{d-1}/2$ (the distance between the center of the hypercube $[0,1]^d$ and the midpoint of any of its edges), then the volume of $H\cap[0,1]^d$ is maximal precisely when $H$ is orthogonal to a diagonal of $[0,1]^d$. It was further shown by Hermann K{\"o}nig \cite{Konig2021} that, if $d$ is at least $5$ and
$$
\frac{\sqrt{d}}{2}-\frac{1}{\sqrt{d}}<t<\frac{\sqrt{d}}{2}\mbox{,}
$$
then that volume is strictly locally maximal when $H$ is orthogonal to a diagonal of $[0,1]^d$. These higher dimensional results are based on an expression for the $(d-1)$-volume of $H\cap[0,1]^d$ that holds when $H$ separates a single vertex of the hypercube from all of its other vertices. A generalization of that expression to arbitrary hyperplanes is established in \cite{Pournin2021}. This formula, an alternative to the improper integral form, is a sum over the vertices of the hypercube on one side of $H$ (see Theorem \ref{LEH.sec.1.thm.2} below) that made it possible to extend the above mentioned result from \cite{MoodyStoneZachZvavitch2013} to when $d\geq5$ and $t$ is greater than $\sqrt{d-2}/2$ (the distance from the center of $[0,1]^d$ to that of a square face) \cite{Pournin2021}.

That problem has also been considered in the case of convex bodies other than the hypercube. The case of cross-polytopes is studied in \cite{LiuTkocz2020,Konig2021,MeyerPajor1988} and the more general case of balls for the $q$-norms in \cite{Koldobsky2005,MeyerPajor1988}. The corresponding problem for the regular simplices is considered in \cite{Konig2021,Webb1996}. Similar problems regarding the $d$-dimensional volume of the portion of the hypercube between two parallel hyperplanes or for the $(d-2)$-dimensional volume of the boundary of hypercube sections are studied, for instance in \cite{BartheKoldobsky2003,Konig2021,KonigKoldobsky2011}.

In the spirit of Vitali Milman's question, the (strict) local extremality of the $(d-1)$\nobreakdash-dimensional volume of $H\cap[0,1]^d$ is studied here when $H$ is orthogonal to a diagonal or a sub-diagonal of $[0,1]^d$ for all dimensions greater than $3$. The first result of the article deals with the case when $t$ is close to $0$.

\begin{thm}\label{LEH.sec.0.thm.0.2}
If $d\geq4$ and $t$ is close enough to $0$, then the $(d-1)$-dimensional volume of $H\cap[0,1]^d$ has strict local maxima when $H$ is orthogonal to a diagonal of $[0,1]^d$ and to any of its sub-diagonals of order at least $4$.
\end{thm}

In this statement, \emph{close enough to $0$} really means that $t$ belongs to an interval of the form $[0,\varepsilon[$ where $\varepsilon$ is a positive number (that depends on $d$). However, $\varepsilon$ is not explicited as the theorem partly follows from a topological argument. The second result extends the local maximality theorem from \cite{Konig2021}. Observe that each of the higher dimensional results mentioned above are stated for ranges of values of $t$ that go to $0$ when $d$ goes to infinity. In the following theorem, the corresponding range is asymptotic to $\sqrt{d}/\!\log{d}$.

\begin{thm}\label{LEH.sec.0.thm.0.1}
If $d\geq4$ and $t$ satisfies
$$
\frac{\sqrt{d}}{2}-\frac{1}{\sqrt{d}}\min\!\left\{\frac{d-1}{4},\frac{d^{1/(d-3)}}{d^{1/(d-3)}-1}\right\}\!<t<\frac{\sqrt{d}}{2}\mbox{,}
$$
then the $(d-1)$-dimensional volume of $H\cap[0,1]^d$ has a strict local maximum when $H$ is orthogonal to a diagonal of $[0,1]^d$.
\end{thm}

While Theorem \ref{LEH.sec.0.thm.0.2} is stated indifferently for the diagonals of the hypercube and its lower order sub-diagonals, Theorem \ref{LEH.sec.0.thm.0.1} does not give a hint of what happens at sub-diagonals. According to the third result of the article, they behave in a very different way than the diagonals.

\begin{thm}\label{LEH.sec.0.thm.0.3}
If $4\leq{n}<d$ and $t$ satisfies
\begin{equation}\label{LEH.sec.0.thm.0.3.eq.1}
\frac{\sqrt{n}}{2}-\frac{1}{\sqrt{n}}\min\!\left\{\frac{n-1}{4},\frac{n^{1/(n-3)}}{n^{1/(n-3)}-1}\right\}\!<t<\frac{\sqrt{n}}{2}\mbox{,}
\end{equation}
then the $(d-1)$-dimensional volume of $H\cap[0,1]^d$ does not have a local extremum when $H$ is orthogonal to an order $n$ sub-diagonal of $[0,1]^d$.
\end{thm}

In fact, Theorems \ref{LEH.sec.0.thm.0.2}, \ref{LEH.sec.0.thm.0.1}, and \ref{LEH.sec.0.thm.0.3} are obtained as consequences of more general results, valid for all possible values of $t$. While these results can be stated in different ways (largely because of the different possible expressions for the volume of $H\cap[0,1]^d$), two statements are particularly noteworthy. In these statements, $p_{i,d}$ is the quadratic function of $z$ defined as
\begin{equation}\label{LEH.sec.0.thm.1.eq.-1}
p_{i,d}(z)=\frac{i(d-i)}{d-1}-\!\left(\frac{d}{2}-i\right)\!\frac{z-i}{d-2}+\frac{2d(z-i)^2}{(d-1)(d-2)}\mbox{.}
\end{equation}

The local extremality for the $(d-1)$-dimensional volume of $H\cap[0,1]^d$ when $H$ is orthogonal to the diagonals of $[0,1]^d$ can be obtained as follows.

\begin{thm}\label{LEH.sec.0.thm.1}
Assume that $d\geq4$. If
\begin{equation}\label{LEH.sec.0.thm.1.eq.0}
\sum_{i=0}^{\lfloor{z}\rfloor}(-1)^i{d\choose{i}}(z-i)^{d-3}p_{i,d}(z)
\end{equation}
is negative, where
$$
z=\frac{d}{2}-t\sqrt{d}\mbox{,}
$$
then the $(d-1)$-dimensional volume of $H\cap[0,1]^d$ has a strict local maximum when $H$ is orthogonal to a diagonal of the hypercube $[0,1]^d$. If, however, (\ref{LEH.sec.0.thm.1.eq.0}) is positive, then the $(d-1)$-dimensional volume of $H\cap[0,1]^d$ has a strict local minimum when $H$ is orthogonal to a diagonal of $[0,1]^d$.
\end{thm}

Because of its piecewise-polynomial nature, (\ref{LEH.sec.0.thm.1.eq.0}) can only possibly vanish at finitely-many values of $z$. Therefore, when $t$ ranges within the interval $[0,\sqrt{d}/2[$, the $(d-1)$\nobreakdash-dimensional volume of $H\cap[0,1]^d$ is almost always strictly locally extremal when $H$ is orthogonal to a diagonal of $[0,1]^d$.

In practice, for any fixed and reasonably low dimension $d$, Theorem \ref{LEH.sec.0.thm.1} allows to completely determine how the local extremality of that volume varies when $H$ is orthogonal to a diagonal of $[0,1]^d$ over the whole interval $[0,\sqrt{d}/2[$ for $t$. It suffices to estimate the roots of $\lceil{d/2}\rceil$ polynomials, each of degree $d-1$. This will be illustrated at the end of the article. A theorem similar to Theorem \ref{LEH.sec.0.thm.1} is established for lower order sub-diagonals.

\begin{thm}\label{LEH.sec.0.thm.2}
Assume that $4\leq{n}<d$. If
\begin{equation}\label{LEH.sec.0.thm.2.eq.0}
\sum_{i=0}^{\lfloor{z}\rfloor}(-1)^i{n\choose{i}}(z-i)^{n-3}p_{i,n}(z)
\end{equation}
and
\begin{equation}\label{LEH.sec.0.thm.2.eq.1}
\sum_{i=0}^{\lfloor{z}\rfloor}(-1)^i{n\choose{i}}(z-i)^{n-3}
\end{equation}
are both negative, where
$$
z=\frac{n}{2}-t\sqrt{n}\mbox{,}
$$
then the $(d-1)$-dimensional volume of $H\cap[0,1]^d$ has a strict local maximum when $H$ is orthogonal to an order $n$ sub-diagonal of $[0,1]^d$. If on the contrary, (\ref{LEH.sec.0.thm.2.eq.0}) and (\ref{LEH.sec.0.thm.2.eq.1}) are both positive, then that volume has a strict local minimum when $H$ is orthogonal to an order $n$ sub-diagonal of $[0,1]^d$.
\end{thm}

Another theorem will be proven as well, providing a condition (roughly, that (\ref{LEH.sec.0.thm.2.eq.0}) and (\ref{LEH.sec.0.thm.2.eq.1}) have opposite signs) under which the considered volume is not strictly locally extremal at the sub-diagonals of the hypercube. Together with Theorem \ref{LEH.sec.0.thm.2}, it allows determine in practice how the local extremality of the $(d-1)$\nobreakdash-dimensional volume of $H$ varies when $H$ is orthogonal to a sub-diagonal of $[0,1]^d$ of fixed, reasonably low order $n$ when $t$ ranges within the whole interval $[0,\sqrt{n}/2[$. As for the diagonals of the hypercube, this will be done at the end of the article relying, in part on symbolic computations.

In Section \ref{LEH.sec.1}, the above mentioned formulas for the $(d-1)$-dimensional volume of $H\cap[0,1]^d$ are recalled, and their partial derivatives with respect to the orientation of $H$ expressed in the two possible forms of an improper integral and a discrete sum over the vertices of the hypercube. The regularity properties of this volume as a function of the orientation of $H$ will be established using the improper integral form, which is the reason why most results in this article are given for dimensions at least $4$ or sub-diagonals of order at least $4$ (as is apparent from the statement of the above theorems). Theorem~\ref{LEH.sec.0.thm.0.2} is proven in Section~\ref{LEH.sec.2} using the second order sufficiency conditions of the constrained Lagrange multipliers theorem and the improper integral form of the partial derivatives of the volume of $H\cap[0,1]^d$. Theorems \ref{LEH.sec.0.thm.0.1} and \ref{LEH.sec.0.thm.1} are both established in Section \ref{LEH.sec.3} using the same Lagrange multipliers strategy, but with the discrete sum form of the partial derivatives. Section \ref{LEH.sec.3.5} is devoted to studying local extremality at the sub-diagonals of the hypercube when $t$ is large or equivalently, when $H$ is far away from the center of the hypercube. Theorems \ref{LEH.sec.0.thm.0.3} and \ref{LEH.sec.0.thm.2} are proven in that section using the second order necessary conditions of the constrained Lagrange multipliers theorem. Finally, the local extremality of the $(d-1)$\nobreakdash-dimensional volume of $H\cap[0,1]^d$ is studied in Section \ref{LEH.sec.4} when $t$ ranges within the whole interval $[0,\sqrt{d}/2[$ at the diagonals of low-dimensional hypercubes and at low order sub-diagonals of hypercubes of arbitrary dimension. A simple, consistent behavior is observed in these cases that is likely to carry over to higher dimensions and sub-diagonal orders.

\section{Partial derivatives of section volumes}\label{LEH.sec.1}

From now on, $a$ denotes a non-zero vector from $\mathbb{R}^d$ and $b$ is a real number. Denote by $H$ the hyperplane of $\mathbb{R}^d$ made up of the points $x$ such that $a\mathord{\cdot}x=b$. The following well-known theorem (see, for instance \cite{Ball1986,Berger2010,FrankRiede2012,KonigKoldobsky2011,Polya1913,Zong2006}) provides an expression for the $(d-1)$-dimensional volume of the intersection $H\cap[0,1]^d$ that takes the form of an improper integral. 
In the sequel, $\sigma(x)$ denotes the sum of the coordinates of a vector $x$ of $\mathbb{R}^d$.

\begin{thm}\label{LEH.sec.1.thm.1}
Assume that $a$ has at least two non-zero coordinates. In that case, the $(d-1)$\nobreakdash-dimensional volume of $H\cap[0,1]^d$ is
$$
\frac{\|a\|}{\pi}\int_{-\infty}^{+\infty}\!\left(\prod_{i=1}^d\frac{\sin(a_iu)}{a_iu}\right)\!\cos\!\left(2\!\left[b-\frac{\sigma(a)}{2}\right]\!u\right)\!du\mbox{.}
$$
\end{thm}

Note that terms of the form $\sin(x)/x$ appear in this expression that are indeterminate when $x$ is equal to $0$. However, under the convention that
$$
\frac{\sin(0)}{0}=1\mbox{,}
$$
which is adopted in the sequel, $\sin(x)/x$ becomes a twice continuously differentiable function of $x$ on $\mathbb{R}$. It will be useful to keep in mind that
$$
\frac{d}{dx}\frac{\sin(x)}{x}
$$
vanishes when $x=0$ and that
$$
\frac{d^2}{dx^2}\frac{\sin(x)}{x}
$$
is equal to $-1/3$ when $x=0$.

When all the coordinates of $a$ are non-zero, the $(d-1)$\nobreakdash-dimensional volume of $H\cap[0,1]^d$ can alternatively be expressed as a sum over a subset of the vertices of the hypercube. This alternative expression, stated in the following theorem, is proven in \cite{Pournin2021} as a straightforward consequence of a result from \cite{BarrowSmith1979}. From now on, $\pi(x)$ denotes the product of the non-zero coordinates of a vector $x$ of $\mathbb{R}^d$. This notation differs from the one used in \cite{Pournin2021}, where $\pi(x)$ denotes the product of all the coordinates of $x$ and not only the non-zero ones. This distinction does not play a role in the following statement, but it will later on.
\begin{thm}\label{LEH.sec.1.thm.2}
If $d$ is at least $2$ and all the coordinates of $a$ are non-zero, then the $(d-1)$-dimensional volume of $H\cap[0,1]^d$ is
\begin{equation}\label{HS.sec.1.thm.2.eq.0}
\sum\frac{(-1)^{\sigma(v)}\|a\|(b-a\mathord{\cdot}v)^{d-1}}{(d-1)!\pi(a)}\mbox{,}
\end{equation}
where the sum is over the vertices $v$ of $[0,1]^d$ such that $a\mathord{\cdot}v\leq{b}$.
\end{thm}

While Theorems \ref{LEH.sec.1.thm.1} and \ref{LEH.sec.1.thm.2} are valid for arbitrary $a$ and $b$, it will be assumed from now on that $a$ belongs to $[0,+\infty[^d\mathord{\setminus}\{0\}$ in order to simplify the analysis. Note that this is without loss of generality thanks to the symmetries of the hypercube. In addition, $b$ will be expressed as
\begin{equation}\label{LEH.sec.1.eq.1}
b=\frac{\sigma(a)}{2}-t\mbox{,}
\end{equation}
where $t$ is a fixed number satisfying
$$
0\leq{t}<\frac{\sqrt{d}}{2}\mbox{.}
$$

In particular, $b$ will be thought of in the sequel as a function of $a$. It will be important to keep in mind that, while $t$ controls the distance between $H$ and the center of $[0,1]^d$, it only coincides with that distance when $\|a\|=1$. From now on, the $(d-1)$-dimensional volume of $H\cap[0,1]^d$ is denoted by $V$ and, just as $b$, this volume is thought of as a function of $a$ on $[0,+\infty[^d\mathord{\setminus}\{0\}$. In \cite{Pournin2021}, this function is shown to be continuous at every point of $[0,+\infty[^d$ with at least two non-zero coordinates and twice continuously differentiable on the open orthant $]0,+\infty[^d$. Here, the following stronger statement will be needed, that is obtained as a consequence of Theorem \ref{LEH.sec.1.thm.1}.

\begin{cor}\label{LEH.sec.1.cor.1}
If $d\geq3$, then $V$ is a continuously differentiable function of $a$ at every point of $[0,+\infty[^d$ with at least three non-zero coordinates. Moreover, if $j$ is an integer satisfying $1\leq{j}\leq{d}$ then, at any such point,
$$
\frac{\partial}{\partial{a_j}}\frac{V}{\|a\|}=\frac{1}{\pi}\int_{-\infty}^{+\infty}\frac{\partial}{\partial{a_j}}\!\left(\prod_{i=1}^d\frac{\sin(a_iu)}{a_iu}\right)\!\cos(2tu)du\mbox{.}
$$
\end{cor}
\begin{proof}
Since $\|a\|$ is a continuously differentiable function of $a$ on $\mathbb{R}^d$, it suffices to show that the partial derivatives of $V/\|a\|$ all exist and are continuous functions of $a$ at the considered points. Thanks to the symmetries of the hypercube, one just needs to prove the slightly stronger statement that the partial derivative of $V/\|a\|$ with respect to $a_3$ exists and is a continuous function of $a$ at every point of $\mathbb{R}^d$ whose first two coordinates are positive.

According to Theorem \ref{LEH.sec.1.thm.1} and to (\ref{LEH.sec.1.eq.1}), at any such point,
$$
\frac{V}{\|a\|}=\int_{-\infty}^{+\infty}\!\left(\prod_{i=1}^d\frac{\sin(a_iu)}{a_iu}\right)\!\cos(2tu)du\mbox{.}
$$

The result will be obtained as a consequence from Leibniz's rule on differentiation under the integral, according to which
$$
\frac{\partial}{\partial{a_3}}\int_{-\infty}^{+\infty}\!\left(\prod_{i=1}^d\frac{\sin(a_iu)}{a_iu}\right)\!\cos(2tu)du=\int_{-\infty}^{+\infty}\frac{\partial}{\partial{a_3}}\!\left(\prod_{i=1}^d\frac{\sin(a_iu)}{a_iu}\right)\!\cos(2tu)du\mbox{.}
$$

However, this rule requires that
\begin{equation}\label{LEH.sec.1.cor.1.eq.1}
\int_{-n}^{+n}\frac{\partial}{\partial{a_3}}\!\left(\prod_{i=1}^d\frac{\sin(a_iu)}{a_iu}\right)\!\cos(2tu)du
\end{equation}
converges uniformly when $n$ goes to infinity in a neighborhood of the considered point from $]0,+\infty[^2\mathord{\times}\mathbb{R}^{d-2}$. Before proceeding with the proof of this uniform convergence property, observe that this will not only provide the existence of the partial derivative but also its continuity. Consider a closed, $d$-dimensional ball $B$ centered at a point contained in $]0,+\infty[^2\mathord{\times}\mathbb{R}^{d-2}$. Pick the radius of $B$ small enough so that it is entirely contained in $]0,+\infty[^2\mathord{\times}\mathbb{R}^{d-2}$. As $\sin(x)/x$ and its derivative with respect to $x$ are bounded functions of $x$ on $\mathbb{R}$, there exists a positive number $M$ such that, for all $u$ in $\mathbb{R}$ and all $a$ in $B$,
$$
\left|\frac{\partial}{\partial{a_3}}\!\left(\prod_{i=1}^d\frac{\sin(a_iu)}{a_iu}\right)\!\cos(2tu)\right|\!\leq{M}\!\left|\frac{\sin(a_1u)\sin(a_2u)}{a_1a_2u^2}\right|\!\mbox{.}
$$

In turn, for any $u$ in $\mathbb{R}$ and $a$ in $B$,
$$
\left|\frac{\sin(a_1u)\sin(a_2u)}{a_1a_2u^2}\right|\!\leq\min\!\left\{1,\frac{1}{mu^2}\right\}\!\mbox{,}
$$
where $m$ is the smallest possible value for the product of the first two coordinates of a point contained in $B$. Therefore, by Cauchy's criterion, (\ref{LEH.sec.1.cor.1.eq.1}) converges uniformly on $B$ when $n$ goes to infinity, as desired.
\end{proof}

By a similar argument as for the proof of Corollary (\ref{LEH.sec.1.cor.1}), one can prove the following, also a consequence of Theorem \ref{LEH.sec.1.thm.1}.

\begin{cor}\label{LEH.sec.1.cor.2}
If $d\geq4$, then $V$ is a twice continuously differentiable function of $a$ at every point of $[0,+\infty[^d$ with at least four non-zero coordinates. Moreover, if $j$ and $k$ are integers satisfying $1\leq{j}\leq{k}\leq{d}$ then, at any such point,
$$
\frac{\partial^2}{\partial{a_j}\partial{a_k}}\frac{V}{\|a\|}=\frac{1}{\pi}\int_{-\infty}^{+\infty}\frac{\partial^2}{\partial{a_j}\partial{a_k}}\!\left(\prod_{i=1}^d\frac{\sin(a_iu)}{a_iu}\right)\!\cos(2tu)du\mbox{.}
$$
\end{cor}
\begin{proof}
As in the proof of Corollary \ref{LEH.sec.1.cor.1}, by the symmetries of the hypercube, it suffices to show the slightly stronger statement that the partial derivative
$$
\frac{\partial^2}{\partial{a_3}\partial{a_k}}\frac{V}{\|a\|}
$$
exists and is a continuous function of $a$ when $k$ is equal to $3$ or to $4$ at every point of $\mathbb{R}^d$ whose first two coordinates are positive. This follows from the same argument than for Corollary \ref{LEH.sec.1.cor.1}, by the observation that $\sin(x)/x$ and its first two derivatives with respect to $x$ are bounded functions of $x$ on $\mathbb{R}$.
\end{proof}

The remainder of the section is devoted to establishing alternative expressions for the partial derivatives of $V/\|a\|$ using Theorem \ref{LEH.sec.1.thm.2} instead of Theorem \ref{LEH.sec.1.thm.1}. From now on, given a subset $X$ of $\mathbb{R}$ and an integer $n$ such that $1\leq{n}\leq{d-1}$, the cartesian product $X^n$ is identified with the subset of $X^n\mathord{\times}\mathbb{R}^{d-n}$ made up of the points whose last $d-n$ coordinates are equal to $0$. 

\begin{lem}\label{LEH.sec.1.lem.1}
Consider an integer $n$ satisfying $3\leq{n}\leq{d}$ and an integer $j$. If $1\leq{j}\leq{n}$, then at any point $a$ from $]0,+\infty[^n$,
$$
\frac{\partial}{\partial{a_j}}\frac{V}{\|a\|}=\sum\frac{(-1)^{\sigma(v)}}{(n-1)!}\frac{\partial}{\partial{a_j}}\frac{(b-a\mathord{\cdot}v)^{n-1}}{\pi(a)}\mbox{,}
$$
where the sum is over the vertices $v$ of $[0,1]^n$ that satisfy $a\mathord{\cdot}v\leq{b}$. If, however $n<j\leq{d}$ then, at any point $a$ from $]0,+\infty[^n$,
$$
\frac{\partial}{\partial{a_j}}\frac{V}{\|a\|}=0\mbox{.}
$$
\end{lem}
\begin{proof}
Note that, when $a$ is a point in $]0,+\infty[^n$, the $n$-dimensional volume of $H\cap[0,1]^n$ coincides with $V$. As the first $n$ coordinates of $a$ are non-zero, one therefore obtains from Theorem~\ref{LEH.sec.1.thm.2} that
\begin{equation}\label{LEH.sec.1.lem.1.eq.1}
\frac{V}{\|a\|}=\sum\frac{(-1)^{\sigma(v)}(b-a\mathord{\cdot}v)^{n-1}}{(n-1)!\pi(a)}\mbox{,}
\end{equation}
where the sum is over the vertices $v$ of $[0,1]^n$ satisfying $a\mathord{\cdot}v\leq{b}$.

Assume that $1\leq{j}\leq{n}$ and observe that, if a vertex $v$ of the hypercube $[0,1]^n$ satisfies $a\mathord{\cdot}v=b$, then the partial derivative
$$
\frac{\partial}{\partial{a_j}}\frac{(b-a\mathord{\cdot}v)^{n-1}}{\pi(a)}
$$
vanishes at $a$ because $n\geq3$. The desired expression for the partial derivative of $V/\|a\|$ with respect to $a_j$ therefore immediately follows from (\ref{LEH.sec.1.lem.1.eq.1}).

Now assume that $n<j\leq{d}$ and recall that $n\geq3$. Hence, according to Corollary \ref{LEH.sec.1.cor.1}, at any point $a$ of the orthant $]0,+\infty[^n$,
$$
\frac{\partial}{\partial{a_j}}\frac{V}{\|a\|}=\frac{1}{\pi}\int_{-\infty}^{+\infty}\frac{\partial}{\partial{a_j}}\!\left(\prod_{i=1}^d\frac{\sin(a_iu)}{a_iu}\right)\!\cos(2tu)du\mbox{.}
$$

Recall that, when $x$ is equal to $0$,
$$
\frac{d}{dx}\frac{\sin(x)}{x}
$$
vanishes. As $a_j=0$, for any point $a$ in $]0,+\infty[^n$,
$$
\frac{\partial}{\partial{a_j}}\!\left(\prod_{i=1}^d\frac{\sin(a_iu)}{a_iu}\right)=0
$$
at this point and for any $u$ in $\mathbb{R}$. Hence, the partial derivative of $V/\|a\|$ with respect to $a_j$ vanishes at any point of $]0,+\infty[^n$, as desired.
\end{proof}

The following lemma is proven using a similar argument.

\begin{lem}\label{LEH.sec.1.lem.2}
Consider an integer $n$ satisfying $4\leq{n}\leq{d}$ and two integers $j$ and $k$. If $1\leq{j}\leq{k}\leq{n}$, then at any point $a$ from $]0,+\infty[^n$,
$$
\frac{\partial^2}{\partial{a_j}\partial{a_k}}\frac{V}{\|a\|}=\sum\frac{(-1)^{\sigma(v)}}{(n-1)!}\frac{\partial^2}{\partial{a_j}\partial{a_k}}\frac{(b-a\mathord{\cdot}v)^{n-1}}{\pi(a)}\mbox{,}
$$
where the sum is over the vertices $v$ of $[0,1]^n$ that satisfy $a\mathord{\cdot}v\leq{b}$. If however, $1\leq{j}<k$ and $n<k\leq{d}$ then, at any such point,
$$
\frac{\partial^2}{\partial{a_j}\partial{a_k}}\frac{V}{\|a\|}=0\mbox{.}
$$
\end{lem}
\begin{proof}
As in the proof of Lemma \ref{LEH.sec.1.lem.1}, when $a$ is a point from $]0,+\infty[^n$,
\begin{equation}\label{LEH.sec.1.lem.2.eq.1}
\frac{V}{\|a\|}=\sum\frac{(-1)^{\sigma(v)}(b-a\mathord{\cdot}v)^{n-1}}{(n-1)!\pi(a)}\mbox{,}
\end{equation}
where the sum is over the vertices $v$ of $[0,1]^n$ satisfying $a\mathord{\cdot}v\leq{b}$.
Moreover, when $1\leq{j}\leq{k}\leq{n}$, and $v$ is a vertex of the hypercube $[0,1]^n$ such that $a\mathord{\cdot}v=b$,
$$
\frac{\partial^2}{\partial{a_j}\partial{a_k}}\frac{(b-a\mathord{\cdot}v)^{n-1}}{\pi(a)}
$$
vanishes at $a$ because $n\geq4$ and the result follows from (\ref{LEH.sec.1.lem.2.eq.1}).

Now assume that $1\leq{j}<k$ and $n<k\leq{d}$. In that case,
$$
\frac{\partial^2}{\partial{a_j}\partial{a_k}}\!\left(\prod_{i=1}^d\frac{\sin(a_iu)}{a_iu}\right)
$$
vanishes at any point $a$ contained in $]0,+\infty[^n$ and for any number $u$ in $\mathbb{R}$ because for any such point, $a_k=0$ and, therefore
$$
\frac{\partial}{\partial{a_k}}\frac{\sin(a_ku)}{a_ku}=0\mbox{.}
$$

The result then immediately follows from Corollary \ref{LEH.sec.1.cor.2}.
\end{proof}

Observe that Lemma \ref{LEH.sec.1.lem.2} provides all the second order partial derivatives of $V/\|a\|$, except the ones with respect to $a_j$ in the case when $j$ is greater than $n$. An expression for them can be obtained from a different argument.

\begin{lem}\label{LEH.sec.1.lem.3}
If $n$ and $j$ are two integers satisfying $4\leq{n}\leq{d}$ and $n<j\leq{d}$, then at any point $a$ contained in $]0,+\infty[^n$,
$$
\frac{\partial^2}{\partial{a_j^2}}\frac{V}{\|a\|}=\sum\frac{(-1)^{\sigma(v)}(b-a\mathord{\cdot}v)^{n-3}}{12(n-3)!\pi(a)}
$$
where the sum is over the vertices $v$ of $[0,1]^n$ that satisfy $a\mathord{\cdot}v\leq{b}$.
\end{lem}
\begin{proof}
By symmetry, it suffices to prove the lemma in the case when $j=n+1$. Consider a point $a$ in $]0,+\infty[^n$. The vertex set of $[0,1]^{n+1}$ can be decomposed into the subset $\mathcal{V}^-$ of the vertices $v$ such that $a\mathord{\cdot}v-b$ is negative, the subset $\mathcal{V}^+$ of the vertices $v$ such that this quantity is positive, and the (possibly empty) subset of the remaining vertices, for which this quantity is equal to $0$.

For any vertex $v$ of the $(n+1)$-dimensional hypercube $[0,1]^{n+1}$, consider the set $U^-_v$ of the points $x$ contained in $\mathbb{R}^{n+1}$ such that $h_v(x)<0$ and the set $U^+_v$ of the points $x$ in $\mathbb{R}^{n+1}$ satisfying $h_v(x)>0$ where
$$
h_v(x)=\frac{\sigma(x)}{2}-t-x\mathord{\cdot}v\mbox{.}
$$

Observe that both $U^-_v$ and $U^+_v$ are open subsets of $\mathbb{R}^{n+1}$. Therefore,
$$
U=\left[\bigcap_{v\in\mathcal{V}^-}\!\!U^-_v\right]\cap\left[\bigcap_{v\in\mathcal{V}^+}\!\!U^+_v\right]
$$
is an open subset of $\mathbb{R}^{n+1}$ as well. Note that by definition, $a$ belongs to $U$.

In the remainder of the proof, $x$ denotes a point contained in $U\cap]0,+\infty[^{n+1}$ that shares its first $n$ coordinates with $a$. In particular, $\pi(x)=\pi(a)x_j$. As none of the first $n+1$ coordinates of $x$ is equal to $0$ and as $n$ is at least $4$, it follows from Theorem \ref{LEH.sec.1.thm.2} and Corollary \ref{LEH.sec.1.cor.2} that
$$
\begin{array}{rcl}
\displaystyle\frac{\partial^2}{\partial{a_j^2}}\frac{V}{\|a\|} & \!\!\!\!=\!\!\!\! & \displaystyle\lim_{x_j\rightarrow0}\frac{\partial^2}{\partial{x_j^2}}\sum\frac{(-1)^{\sigma(v)}[h_v(x)]^n}{n!\pi(x)}\mbox{,}\\[\bigskipamount]
 & \!\!\!\!=\!\!\!\! & \displaystyle\lim_{x_j\rightarrow0}\sum\frac{(-1)^{\sigma(v)}}{n!\pi(a)}\frac{\partial^2}{\partial{x_j^2}}\frac{[h_v(x)]^n}{x_j}\mbox{,}\\
\end{array}
$$
where the sums are over the elements $v$ of the union of $\mathcal{V}^+$ with a (possibly empty) subset of $\mathcal{V}^\circ$. Note that, for any vertex $v$ of $[0,1]^{n+1}$,
\begin{multline}\label{LEH.sec.1.lem.3.eq.0}
\frac{\partial^2}{\partial{x_j^2}}\frac{[h_v(x)]^n}{x_j}=2\frac{[h_v(x)]^n}{x_j^3}-2n\left(\frac{1}{2}-v_j\right)\frac{[h_v(x)]^{n-1}}{x_j^2}\\
\hfill+\frac{n(n-1)}{4}\frac{[h_v(x)]^{n-2}}{x_j}\mbox{.}
\end{multline}

However, one obtains from l'H{\^o}pital's rule that, when $v$ belongs to $\mathcal{V}^\circ$,
$$
\lim_{x_j\rightarrow0}\frac{[h_v(x)]^n}{x_j^3}=\lim_{x_j\rightarrow0}\frac{[h_v(x)]^{n-1}}{x_j^2}=\lim_{x_j\rightarrow0}\frac{[h_v(x)]^{n-2}}{x_j}=\lim_{x_j\rightarrow0}[h_v(x)]^{n-3}=0
$$
because $n\geq4$ and $h_v(a)=0$. As a consequence,
\begin{equation}\label{LEH.sec.1.lem.3.eq.1}
\frac{\partial^2}{\partial{a_j^2}}\frac{V}{\|a\|}=\lim_{x_j\rightarrow0}\sum_{v\in\mathcal{V}^+}\frac{(-1)^{\sigma(v)}}{n!\pi(a)}\frac{\partial^2}{\partial{x_j^2}}\frac{[h_v(x)]^n}{x_j}
\end{equation}

Now observe that, since $a_j=0$, a vertex $v$ of $[0,1]^n$ belongs to $\mathcal{V}^+$ if and only if the vertex $w$ of $[0,1]^{n+1}$ that shares its first $n$ coordinates with $v$ but such that $w_j=1$ also belongs to $\mathcal{V}^+$. Hence, gathering the terms corresponding to each such pair of vertices $v$ and $w$, (\ref{LEH.sec.1.lem.3.eq.1}) can be rewritten into
\begin{equation}\label{LEH.sec.1.lem.3.eq.2}
\frac{\partial^2}{\partial{a_j^2}}\frac{V}{\|a\|}=\lim_{x_j\rightarrow0}\sum\frac{(-1)^{\sigma(v)}}{n!\pi(a)}\frac{\partial^2}{\partial{x_j^2}}\frac{[h_v(x)]^n-[h_v(x)-x_j]^n}{x_j}
\end{equation}
where the sum is over the vertices $v$ of $[0,1]^n$ satisfying $a\mathord{\cdot}v<b$. However, for any vertex $v$ of $[0,1]^n$, one obtains from (\ref{LEH.sec.1.lem.3.eq.0}) that
$$
\frac{\partial^2}{\partial{x_j^2}}\frac{[h_v(x)]^n-[h_v(x)-x_j]^n}{x_j}=r_v(x)+s_v(x)
$$
where
$$
r_v(x)=\frac{2[h_v(x)]^n-2[h_v(x)-x_j]^n-nx_j[h_v(x)]^{n-1}-nx_j[h_v(x)-x_j]^{n-1}}{x_j^3}
$$
and
$$
s_v(x)=n(n-1)\frac{[h_v(x)]^{n-2}-[h_v(x)-x_j]^{n-2}}{4x_j}\mbox{.}
$$

It turns out that $r_v(x)$ and $s_v(x)$ both admit limits when $x_j$ goes to $0$. Observe that, in the above expression of $r_v(x)$ as a ratio, both the numerator and the denominator go to zero as $x_j$ goes to $0$. By applying l'H{\^o}pital's rule twice,
$$
\begin{array}{rcl}
\displaystyle\lim_{x_j\rightarrow0}r_v(x) & \!\!\!\!=\!\!\!\! & \displaystyle\lim_{x_j\rightarrow0}-n(n-1)\frac{[h_v(x)]^{n-2}-[h_v(x)-x_j]^{n-2}}{6x_j}\\[\bigskipamount]
 & \!\!\!\!=\!\!\!\! & \displaystyle\lim_{x_j\rightarrow0}-n(n-1)(n-2)\frac{[h_v(x)]^{n-3}+[h_v(x)-x_j]^{n-3}}{12}\\[\bigskipamount]
 & \!\!\!\!=\!\!\!\! & \displaystyle-n(n-1)(n-2)\frac{(b-a\mathord{\cdot}v)^{n-3}}{6}\mbox{.}\\
\end{array}
$$

Similarly, the numerator and the denominator of the expression of $s_v(x)$ as a ratio both go to zero as $x_j$ goes to $0$. By l'H{\^o}pital's rule,
$$
\begin{array}{rcl}
\displaystyle\lim_{x_j\rightarrow0}s_v(x) & \!\!\!\!=\!\!\!\! & \displaystyle\lim_{x_j\rightarrow0}n(n-1)(n-2)\frac{[h_v(x)]^{n-3}+[h_v(x)-x_j]^{n-3}}{8}\\[\bigskipamount]
 & \!\!\!\!=\!\!\!\! & \displaystyle{n(n-1)(n-2)\frac{(b-a\mathord{\cdot}v)^{n-3}}{4}}\mbox{.}\\
\end{array}
$$

As a consequence, (\ref{LEH.sec.1.lem.3.eq.2}) yields
$$
\frac{\partial^2}{\partial{a_j^2}}\frac{V}{\|a\|}=\sum\frac{(-1)^{\sigma(v)}(b-a\mathord{\cdot}v)^{n-3}}{12(n-3)!\pi(a)}
$$
where the sum is over the vertices $v$ of $[0,1]^n$ such that $a\mathord{\cdot}v<b$. As $n\geq4$, adding to this sum the terms that correspond to the vertices $v$ of $[0,1]^n$ such that $a\mathord{\cdot}v=b$ does not affect it, providing the desired equality.
\end{proof}

\section{Local maxima near the center of the hypercube}\label{LEH.sec.2}

The following general result will be used later in the section in order to establish the local maximality of $V$ when $t$ is close enough to $0$, at the point $a$ whose first $n$ coordinates are $1/\sqrt{n}$ and whose other coordinates are $0$. It will also be used to establish the local maximality results of Section \ref{LEH.sec.3}.

\begin{thm}\label{LEH.sec.2.thm.1}
Consider an integer $n$ satisfying $2\leq{n}\leq{d}$ and assume that $V$ is twice continuously differentiable at the point $a$ of $\mathbb{R}^n$ whose first $n$ coordinates are equal to $1/\sqrt{n}$. If at that point,
\begin{equation}\label{LEH.sec.2.thm.1.eq.0.1}
\frac{\partial^2}{\partial{a_1^2}}\frac{V}{\|a\|}-\sqrt{n}\frac{\partial}{\partial{a_1}}\frac{V}{\|a\|}-\frac{\partial^2}{\partial{a_1}\partial{a_2}}\frac{V}{\|a\|}
\end{equation}
is negative and, when $n<d$,
\begin{equation}\label{LEH.sec.2.thm.1.eq.0.2}
\frac{\partial^2}{\partial{a_d^2}}\frac{V}{\|a\|}
\end{equation}
is also negative, then $V$ has a strict local maximum on $\mathbb{S}^{d-1}\cap[0,+\infty[^d$ at $a$. Similarly, if (\ref{LEH.sec.2.thm.1.eq.0.1}) is positive at $a$ and, when $n<d$, (\ref{LEH.sec.2.thm.1.eq.0.2}) is positive at $a$ as well, then $V$ admits a strict local minimum on $\mathbb{S}^{d-1}\cap[0,+\infty[^d$ at that point.
\end{thm}
\begin{proof}
Consider a number $\lambda$ and denote
$$
L_\lambda=\frac{V}{\|a\|}+\lambda\!\left(\|a\|^2-1\right)\!.
$$

Throughout the proof $L_\lambda$ is treated as a function of $a$. Recall that $a$ is a critical point of $L_\lambda$ when, for all integers $j$ such that $1\leq{j}\leq{d}$,
\begin{equation}\label{LEH.sec.2.thm.1.eq.1}
\frac{\partial{L_\lambda}}{\partial{a_j}}=0\mbox{.}
\end{equation}

From now on, $a$ denotes the point contained in $]0,+\infty[^n$ whose first $n$ coordinates are equal to $1/\sqrt{n}$ and all the (first or second order) partial derivatives are taken at that point. Pick $\lambda$ such that
$$
\lambda=-\frac{\sqrt{n}}{2}\frac{\partial}{\partial{a_1}}\frac{V}{\|a\|}\mbox{.}
$$

Note that this particular value of $\lambda$ satisfies (\ref{LEH.sec.2.thm.1.eq.1}) when $j$ is equal to $1$. By symmetry, the partial derivatives of $L_\lambda$ with respect to $a_j$ all coincide at $a$ when $1\leq{j}\leq{n}$. Hence (\ref{LEH.sec.2.thm.1.eq.1}) holds for these values of $j$. Recall that the last $d-n$ coordinates of $a$ are equal to $0$. Therefore at the point $a$,
$$
\frac{\partial{L_\lambda}}{\partial{a_j}}=\frac{\partial}{\partial{a_j}}\frac{V}{\|a\|}
$$ 
when $n<j\leq{d}$ and it immediately follows from Lemma \ref{LEH.sec.1.lem.1} that (\ref{LEH.sec.2.thm.1.eq.1}) also holds in that case. As a consequence, $a$ is a critical point of $L_\lambda$ and in turn, according to the second order sufficiency conditions of the Lagrange multipliers theorem (see for instance Proposition~3.2.1 in \cite{Bertsekas1999}), if at the point $a$,
\begin{equation}\label{LEH.sec.2.thm.1.eq.2}
\sum_{j=1}^d\sum_{k=1}^dx_jx_k\frac{\partial^2{L_\lambda}}{\partial{a_j}\partial{a_k}}<0
\end{equation}
for every non-zero point $x$ in $\mathbb{R}^d$ whose first $n$ coordinates sum to $0$, then $V$ has a strict local maximum on $\mathbb{S}^{d-1}\cap[0,+\infty[^d$ at that point.

Now observe that, by symmetry, at point $a$,
$$
\frac{\partial^2{L_\lambda}}{\partial{a_j^2}}=\frac{\partial^2}{\partial{a_1^2}}\frac{V}{\|a\|}-\sqrt{n}\frac{\partial}{\partial{a_1}}\frac{V}{\|a\|}
$$
when $1\leq{j}\leq{n}$,
$$
\frac{\partial^2{L_\lambda}}{\partial{a_j}\partial{a_k}}=\frac{\partial^2}{\partial{a_1}\partial{a_2}}\frac{V}{\|a\|}
$$
when $1\leq{j}<k\leq{n}$, and
$$
\frac{\partial^2{L_\lambda}}{\partial{a_j^2}}=\frac{\partial^2}{\partial{a_d^2}}\frac{V}{\|a\|}
$$
when $n<j\leq{d}$. Moreover, according to Lemma \ref{LEH.sec.1.lem.2}, all the other second order partial derivatives vanish. As a consequence, at that point,
\begin{multline*}
\sum_{j=1}^d\sum_{k=1}^dx_jx_k\frac{\partial^2{L_\lambda}}{\partial{a_j}\partial{a_k}}=\!\left[\frac{\partial^2}{\partial{a_1^2}}\frac{V}{\|a\|}-\sqrt{n}\frac{\partial}{\partial{a_1}}\frac{V}{\|a\|}-\frac{\partial^2}{\partial{a_1}\partial{a_2}}\frac{V}{\|a\|}\right]\!\sum_{i=1}^nx_i^2\\
+\frac{\partial^2}{\partial{a_1}\partial{a_2}}\frac{V}{\|a\|}\!\left[\sum_{i=1}^nx_i\right]^2\!+\frac{\partial^2}{\partial{a_d^2}}\frac{V}{\|a\|}\sum_{i=n+1}^dx_i^2\mbox{.}
\end{multline*}

Note that, if the first $n$ coordinates of $x$ sum to $0$, then the second term in the right-hand side of this equality vanishes. It follows that, if (\ref{LEH.sec.2.thm.1.eq.0.1}) is negative at $a$ and, when $n<d$, (\ref{LEH.sec.2.thm.1.eq.0.2}) is also negative at that point, then (\ref{LEH.sec.2.thm.1.eq.2}) holds for every non-zero point $x$ in $\mathbb{R}^d$ whose first $n$ coordinates sum to $0$. Hence, $V$ has a strict local maximum on $\mathbb{S}^{d-1}\cap[0,+\infty[^d$ at $a$, as desired.

Finally, observe that repeating this proof, but with the second order sufficiency conditions of the Lagrange multipliers theorem for local minima (instead of local maxima), one obtains 
the desired local minimality result.
\end{proof}

Using Corollaries \ref{LEH.sec.1.cor.1} and \ref{LEH.sec.1.cor.2}, expressions for (\ref{LEH.sec.2.thm.1.eq.0.1}) and (\ref{LEH.sec.2.thm.1.eq.0.2}) can be obtained as follows in the form of improper integrals.

\begin{lem}\label{LEH.sec.2.lem.1}
Consider an integer $n$ satisfying $4\leq{n}\leq{d}$. At the point $a$ of $\mathbb{R}^n$ whose first $n$ coordinates are all equal to $1/\sqrt{n}$,
\begin{multline}\label{LEH.sec.2.lem.1.eq.1}
\frac{\partial^2}{\partial{a_1^2}}\frac{V}{\|a\|}-\sqrt{n}\frac{\partial}{\partial{a_1}}\frac{V}{\|a\|}-\frac{\partial^2}{\partial{a_1}\partial{a_2}}\frac{V}{\|a\|}=\frac{1}{\pi}\int_{-\infty}^{+\infty}n\Biggl[2\frac{n}{u^2}\sin^2\!\left(\frac{u}{\sqrt{n}}\right)\\
-\frac{\sqrt{n}}{u}\cos\!\left(\frac{u}{\sqrt{n}}\right)\!\sin\!\left(\frac{u}{\sqrt{n}}\right)\!-1\Biggr]\!\Biggl(\frac{\sqrt{n}}{u}\sin\!\left(\frac{u}{\sqrt{n}}\right)\!\Biggr)^{\!\!n-2}\!\!\!\!\!\!\!\!\!\cos(2tu)du
\end{multline}
and, if $n<d$, then
\begin{equation}\label{LEH.sec.2.lem.1.eq.2}
\frac{\partial^2}{\partial{a_d^2}}\frac{V}{\|a\|}=-\frac{1}{3\pi}\int_{-\infty}^{+\infty}u^2\!\Biggl(\frac{\sqrt{n}}{u}\sin\!\left(\frac{u}{\sqrt{n}}\right)\!\Biggr)^{\!\!n}\!\!\!\cos(2tu)du\mbox{.}
\end{equation}
\end{lem}
\begin{proof}
First recall that, when $x$ is non-zero,
$$
\frac{\partial}{\partial{x}}\frac{\sin(x)}{x}=\frac{\cos(x)}{x}-\frac{\sin(x)}{x^2}\mbox{.}
$$

Moreover, this partial derivative vanishes when $x$ is equal to $0$. Therefore, according to Corollary \ref{LEH.sec.1.cor.1}, when $a$ belongs to $]0,+\infty[^n$,
$$
\displaystyle\frac{\partial}{\partial{a_1}}\frac{V}{\|a\|}=\frac{1}{\pi}\int_{-\infty}^{+\infty}u\Biggl[\frac{\cos(a_1u)}{a_1u}-\frac{\sin(a_1u)}{a_1^2u^2}\Biggr]\!\!\left(\prod_{i=2}^n\frac{\sin(a_iu)}{a_iu}\right)\!\cos(2tu)du\mbox{.}
$$

Hence, at the point $a$ of $\mathbb{R}^n$ whose first $n$ coordinates are $1/\sqrt{n}$,
\begin{multline}\label{LEH.sec.2.lem.1.eq.3}
\displaystyle\frac{\partial}{\partial{a_1}}\frac{V}{\|a\|}=\frac{1}{\pi}\int_{-\infty}^{+\infty}u\Biggl[\frac{\sqrt{n}}{u}\cos\!\left(\frac{u}{\sqrt{n}}\right)\!\\
\hfill-\frac{n}{u^2}\sin\!\left(\frac{u}{\sqrt{n}}\right)\!\Biggr]\!\Biggl(\frac{\sqrt{n}}{u}\sin\!\left(\frac{u}{\sqrt{n}}\right)\!\Biggr)^{\!\!n-1}\!\!\!\!\!\!\!\!\!\cos(2tu)du\mbox{.}
\end{multline}

Further recall that, when $x$ is not equal to $0$
$$
\frac{\partial^2}{\partial{x^2}}\frac{\sin(x)}{x}=-\frac{\sin(x)}{x}-\frac{2\cos(x)}{x^2}+\frac{2\sin(x)}{x^3}
$$
and that this second order partial derivative is equal to $-1/3$ when $x$ is equal to $0$. Hence, by Corollary \ref{LEH.sec.1.cor.2}, when $a\in]0,+\infty[^n$,
\begin{multline*}
\displaystyle\frac{\partial^2}{\partial{a_1^2}}\frac{V}{\|a\|}=\frac{1}{\pi}\int_{-\infty}^{+\infty}u^2\Biggl[-\frac{\sin(a_1u)}{a_1u}-\frac{2\cos(a_1u)}{a_1^2u^2}\\
\hfill+\frac{2\sin(a_1u)}{a_1^3u^3}\Biggr]\!\!\left(\prod_{i=2}^n\frac{\sin(a_iu)}{a_iu}\right)\!\cos(2tu)du
\end{multline*}
and
\begin{multline*}
\displaystyle\frac{\partial^2}{\partial{a_1}\partial{a_2}}\frac{V}{\|a\|}=\frac{1}{\pi}\int_{-\infty}^{+\infty}u^2\Biggl[\frac{\cos(a_1u)\cos(a_2u)}{a_1a_2u^2}-\frac{\cos(a_1u)\sin(a_2u)}{a_1a_2^2u^3}\\
\hfill-\frac{\sin(a_1u)\cos(a_2u)}{a_1^2a_2u^3}+\frac{\sin(a_1u)\sin(a_2u)}{a_1^2a_2^2u^4}\Biggr]\!\!\left(\prod_{i=3}^n\frac{\sin(a_iu)}{a_iu}\right)\!\cos(2tu)du\mbox{.}
\end{multline*}

Therefore, at the point $a$ of $\mathbb{R}^n$ whose first $n$ coordinates are $1/\sqrt{n}$,
\begin{multline}\label{LEH.sec.2.lem.1.eq.4}
\displaystyle\frac{\partial^2}{\partial{a_1^2}}\frac{V}{\|a\|}=\frac{1}{\pi}\int_{-\infty}^{+\infty}u^2\Biggl[-\frac{n}{u^2}\sin^2\!\left(\frac{u}{\sqrt{n}}\right)\!-\frac{2n\sqrt{n}}{u^3}\cos\!\left(\frac{u}{\sqrt{n}}\right)\!\sin\!\left(\frac{u}{\sqrt{n}}\right)\\
\hfill+\frac{2n^2}{u^4}\sin^2\!\left(\frac{u}{\sqrt{n}}\right)\!\Biggr]\!\Biggl(\frac{\sqrt{n}}{u}\sin\!\left(\frac{u}{\sqrt{n}}\right)\!\Biggr)^{\!\!n-2}\!\!\!\!\!\!\!\!\!\cos(2tu)du
\end{multline}
and
\begin{multline}\label{LEH.sec.2.lem.1.eq.5}
\displaystyle\frac{\partial^2}{\partial{a_1}\partial{a_2}}\frac{V}{\|a\|}=\frac{1}{\pi}\int_{-\infty}^{+\infty}u^2\Biggl[\frac{n}{u^2}\cos^2\!\left(\frac{u}{\sqrt{n}}\right)\!-\frac{2n\sqrt{n}}{u^3}\cos\!\left(\frac{u}{\sqrt{n}}\right)\!\sin\!\left(\frac{u}{\sqrt{n}}\right)\\
\hfill+\frac{n^2}{u^4}\sin^2\!\left(\frac{u}{\sqrt{n}}\right)\!\Biggr]\!\Biggl(\frac{\sqrt{n}}{u}\sin\!\left(\frac{u}{\sqrt{n}}\right)\!\Biggr)^{\!\!n-2}\!\!\!\!\!\!\!\!\!\cos(2tu)du\mbox{.}
\end{multline}

Combining (\ref{LEH.sec.2.lem.1.eq.3}), (\ref{LEH.sec.2.lem.1.eq.4}), and (\ref{LEH.sec.2.lem.1.eq.5}) yields (\ref{LEH.sec.2.lem.1.eq.1}). Finally, assume that $n<d$. As the second order derivative of $\sin(x)/x$ is equal to $-1/3$ when $x$ is equal to $0$, at the point $a$ of $\mathbb{R}^n$ whose first $n$ coordinates are $1/\sqrt{n}$,
$$
\frac{\partial^2}{\partial{a_d^2}}\prod_{i=1}^d\frac{\sin(a_iu)}{a_iu}=-\frac{u^2}{3}\!\Biggl(\frac{\sqrt{n}}{u}\sin\!\left(\frac{u}{\sqrt{n}}\right)\!\Biggr)^{\!\!n}
$$
and by Corollary \ref{LEH.sec.1.cor.2}, (\ref{LEH.sec.2.lem.1.eq.2}) holds at that point.
\end{proof}

In order to estimate (\ref{LEH.sec.2.lem.1.eq.1}) when $t$ is close to $0$, the following three technical propositions will be needed. A proof of each is provided for completeness.

\begin{prop}\label{LEH.sec.2.prop.1}
The quantity
\begin{equation}\label{LEH.sec.2.prop.1.eq.1}
\frac{1-\cos(2s)}{s^2}-\frac{\sin(2s)}{2s}-1
\end{equation}
is negative when $s$ is positive.
\end{prop}
\begin{proof}
Assume that $s$ is positive and observe that
$$
\frac{1-\cos(2s)}{s^2}-\frac{\sin(2s)}{2s}\leq\frac{2}{s^2}+\frac{1}{2s}\mbox{.}
$$
Hence, the result is immediate when $s$ is greater than $2$. Assume that $s$ is at most $2$. Expanding $\cos(2s)$ and $\sin(2s)$ into their power series yields
$$
\frac{1-\cos(2s)}{s^2}-\frac{\sin(2s)}{2s}-1=-\sum_{i=2}^{+\infty}\frac{(-1)^i(2i-2)}{(2i+2)!}(2s)^{2i}\mbox{.}
$$

Note that in the right-hand side, any two consecutive terms of the sum have opposite signs. It is therefore sufficient to show that
$$
\frac{2i-2}{(2i+2)!}(2s)^{2i}-\frac{2i}{(2i+4)!}(2s)^{2i+2}
$$
is positive when $i$ is an even positive integer. Observe that
$$
\frac{2i-2}{(2i+2)!}(2s)^{2i}-\frac{2i}{(2i+4)!}(2s)^{2i+2}=2\frac{(2s)^{2i}}{(2i+2)!}(iR_i-1)\mbox{,}
$$
where
$$
R_i=1-\frac{4s^2}{(2i+3)(2i+4)}\mbox{.}
$$

As $R_i$ is an increasing function of $i$, one obtains that it is at least $1-s^2/14$ for every even positive integer $i$. In turn, as $s$ is at most $2$ and $i$ at least $2$, this shows that $iR_i-1$ is at least $1-4/7$, which is also positive.
\end{proof}

\begin{prop}\label{LEH.sec.2.prop.2}
If $n$ is an integer greater than  $5$, then
\begin{equation}\label{LEH.sec.2.prop.2.eq.1}
\frac{1-\cos(2s)}{s^n}-\frac{\sin(2s)}{2s^{n-1}}-\frac{1}{s^{n-2}}
\end{equation}
is a strictly increasing function of $s$ on $]0,+\infty[$ and if $n$ is equal to $5$, then it is a strictly increasing function of $s$ on $]4,+\infty[$.
\end{prop}
\begin{proof}
Observe that (\ref{LEH.sec.2.prop.2.eq.1}) is a differentiable function of $s$ on $]0,+\infty[$ and that its derivative with respect to $s$ can be expressed as
$$
\frac{d}{ds}\!\left[\frac{1-\cos(2s)}{s^n}-\frac{\sin(2s)}{2s^{n-1}}-\frac{1}{s^{n-2}}\right]\!=\frac{An+B}{s^{n-1}}
$$
where
$$
A=\frac{\cos(2s)-1}{s^2}+\frac{\sin(2s)}{2s}+1
$$
and
$$
B=\frac{3\sin(2s)}{2s}-\cos(2s)-2\mbox{.}
$$

According to Proposition \ref{LEH.sec.2.prop.1}, $A$ is positive. Therefore, it suffices to show that $6A+B$ is always positive and that $5A+B$ is positive when $s$ is greater than $4$. The latter is immediate. Indeed, observe that
$$
\begin{array}{rcl}
5A+B & \!\!\!\!\geq\!\!\!\! & \displaystyle5\!\left(-\frac{2}{s^2}-\frac{1}{2s}+1\right)\!-\frac{3}{2s}-3\mbox{,}\\[\bigskipamount]
& \!\!\!\!=\!\!\!\! & \displaystyle2-\frac{10}{s^2}-\frac{4}{s}\mbox{.}
\end{array}
$$

It remains to show that $6A+B$ is positive. Observe that
$$
\begin{array}{rcl}
6A+B & \!\!\!\!=\!\!\!\! & \displaystyle\!\left(\frac{6}{s^2}-1\right)\!\cos(2s)-\frac{6}{s^2}+\frac{9\sin(2s)}{2s}+4\\[\bigskipamount]
& \!\!\!\!\geq\!\!\!\! & \displaystyle-\!\left|\frac{6}{s^2}-1\right|\!-\frac{6}{s^2}-\frac{9}{2s}+4\mbox{.}
\end{array}
$$

As a consequence, $6A+B$ is positive when $s$ is at least $\sqrt{6}$. It is assumed in the remainder of the proof that $s$ is less than $\sqrt{6}$. Expanding $\sin(2s)$ and $\cos(2s)$ into their power series, one obtains
$$
6A+B=4\sum_{i=2}^{+\infty}\frac{(-1)^i(i-1)i}{(2i+4)!}(2s)^{2i+2}\mbox{.}
$$

Note that the terms in that sum have alternating signs. Rearranging the sum so that each positive term is summed with the next one yields
\begin{equation}\label{LEH.sec.2.prop.2.eq.2}
6A+B=8\sum_{i=1}^{+\infty}\frac{i(2s)^{4i+2}}{(4i+4)!}\!\left(R_i-2\right)\!\mbox{,}
\end{equation}
where
$$
R_i=(2i+1)\!\left(1-\frac{4s^2}{(4i+5)(4i+6)}\right)\!\!\mbox{.}
$$

Note that $R_i$ is an increasing function of $i$. Hence, as $i$ is a positive integer,
$$
R_i\geq3\!\left(1-\frac{4s^2}{90}\right)\!\!\mbox{.}
$$

Now recall that $s$ is less than $\sqrt{6}$. It follows that $R_i$ is greater than $22/10$ for every positive integer $i$ and, by (\ref{LEH.sec.2.prop.2.eq.2}), that $6A+B$ is positive.
\end{proof}

\begin{prop}\label{LEH.sec.2.prop.3}
$\displaystyle\int_0^{2\pi}\!\!\left(2\frac{\sin^5(s)}{s^5}-\cos(s)\frac{\sin^4(s)}{s^4}-\frac{\sin^3(s)}{s^3}\right)\!ds<0$.
\end{prop}
\begin{proof}
Denote
$$
f(s)=2\frac{\sin^2(s)}{s^5}-\cos(s)\frac{\sin(s)}{s^4}-\frac{1}{s^3}\mbox{.}
$$

Recall that $2\sin^2(s)=1-\cos(2s)$ and $2\sin(s)\cos(s)=\sin(2s)$. Therefore, by Proposition \ref{LEH.sec.2.prop.1}, $f(s)$ is negative when $s$ belongs to $]0,2\pi[$. Hence,
$$
\int_0^{2\pi}f(s)\sin^3(s)ds<\sin^3(1)\!\left[\int_1^{\pi-1}f(s)ds-\int_\pi^{\pi+1}f(s)ds\right]\!-\int_{\pi+1}^{2\pi}f(s)ds\mbox{.}
$$

As in addition,
$$
f(s)=\frac{d}{ds}\frac{2s^2+\cos(2s)-1}{4s^4}
$$
one obtains the inequality
\begin{multline*}
\int_0^{2\pi}f(s)\sin^3(s)ds<\sin^3(1)\!\biggl[\frac{2(\pi-1)^2+\cos(2\pi-2)-1}{4(\pi-1)^4}-\frac{1+\cos(2)}{4}\\
\hfill+\frac{1}{2\pi^2}\biggr]\!-\frac{1}{8\pi^2}+\!\left(1-\sin^3(1)\right)\!\frac{2(\pi+1)^2+\cos(2\pi+2)-1}{4(\pi+1)^4}
\end{multline*}
whose right-hand side is negative.
\end{proof}

Theorem \ref{LEH.sec.0.thm.0.2} can now be established

\begin{proof}[Proof of Theorem \ref{LEH.sec.0.thm.0.2}]
Consider an integer $n$ satisfying $4\leq{n}\leq{d}$. By the symmetries of the hypercube, it is only required to show that, if $t$ is small enough, then $V$ has a local maximum at the point $a$ of $\mathbb{R}^n$ whose first $n$ coordinates are equal to $1/\sqrt{n}$. Observe that the quantities
\begin{multline}
Q_1=\int_{-\infty}^{+\infty}\!\biggl[2\frac{n}{u^2}\sin^2\!\left(\frac{u}{\sqrt{n}}\right)-\frac{\sqrt{n}}{u}\cos\!\left(\frac{u}{\sqrt{n}}\right)\!\sin\!\left(\frac{u}{\sqrt{n}}\right)\!\\
\hfill-1\biggr]\!\!\Biggl(\frac{\sqrt{n}}{u}\sin\!\left(\frac{u}{\sqrt{n}}\right)\!\Biggr)^{\!\!n-2}\!\!\!\!\!\!\!\!\!\cos(2tu)du\mbox{,}
\end{multline}
and
$$
Q_2=-\int_{-\infty}^{+\infty}u^2\!\Biggl(\frac{\sqrt{n}}{u}\sin\!\left(\frac{u}{\sqrt{n}}\right)\!\Biggr)^{\!\!n}\!\!\!\cos(2tu)du
$$
are continuous functions of $t$ on $\mathbb{R}$. Therefore, according to Theorem \ref{LEH.sec.2.thm.1} and Lemma \ref{LEH.sec.2.lem.1}, it suffices to show that both of them are negative when $t$ is equal to~$0$. Assume that $t$ is equal to $0$ and note that, under this assumption, the negativity of $Q_2$ is immediate when $n$ is even.

By the change of variables $s=u/\sqrt{n}$ and splitting the integral at $0$,
$$
Q_1=2\sqrt{n}\int_{0}^{+\infty}\!\left[2\frac{\sin^n(s)}{s^n}-\cos(s)\frac{\sin^{n-1}(s)}{s^{n-1}}-\frac{\sin^{n-2}(s)}{s^{n-2}}\right]\!ds
$$

Now observe that since
\begin{equation}\label{LEH.sec.2.thm.2.eq.1}
2\frac{\sin^2(s)}{s^n}-\cos(s)\frac{\sin(s)}{s^{n-1}}-\frac{1}{s^{n-2}}=\frac{1-\cos(2s)}{s^n}-\frac{\sin(2s)}{2s^{n-1}}-\frac{1}{s^{n-2}}\mbox{,}
\end{equation}
the negativity of $Q_1$ is a consequence of Proposition \ref{LEH.sec.2.prop.1} when $n$ is even. It is therefore assumed for the remainder of the proof that $n$ is odd.

Further splitting the integral at the integer multiples of $\pi$ yields
$$
Q_1=2\sqrt{n}\sum_{i=0}^{+\infty}I_i
$$
where for any non-negative integer $i$,
$$
I_i=\int_{i\pi}^{(i+1)\pi}\!\left[2\frac{\sin^n(s)}{s^n}-\cos(s)\frac{\sin^{n-1}(s)}{s^{n-1}}-\frac{\sin^{n-2}(s)}{s^{n-2}}\right]\!ds\mbox{.}
$$

Hence, in order to show that $Q_1$ is negative, it is sufficient to prove that the sum $I_i+I_{i+1}$ is negative for all even $i$. When $i$ is equal to $0$ and $n$ to $5$, this follows from Proposition \ref{LEH.sec.2.prop.3}. Now observe that as $n$ is odd,
\begin{multline*}
I_i+I_{i+1}=\int_{i\pi}^{(i+1)\pi}\!\Biggl[2\frac{\sin^2(s)}{s^n}-\cos(s)\frac{\sin(s)}{s^{n-1}}-\frac{1}{s^{n-2}}\Biggr]\!\sin^{n-2}(s)\\
-\Biggl[2\frac{\sin^2(s+\pi)}{(s+\pi)^n}-\cos(s+\pi)\frac{\sin(s+\pi)}{(s+\pi)^{n-1}}-\frac{1}{(s+\pi)^{n-2}}\Biggr]\!\sin^{n-2}(s)ds\mbox{.}
\end{multline*}

Hence, according to (\ref{LEH.sec.2.thm.2.eq.1}) and to Proposition \ref{LEH.sec.2.prop.2}, $I_i+I_{i+1}$ is negative when $n$ is equal to $5$ and $i$ is a positive even number. It is also negative when $n$ greater than $5$ for any non-negative even integer $i$. This shows that $Q_1$ is negative and it remains to show that $Q_2$ is also negative.

By the change of variables $s=u/\sqrt{n}$ in the expression of $Q_2$,
$$
Q_2=-2n^{3/2}\sum_{i=0}^{+\infty}J_i
$$
where, for any non-negative integer $i$,
$$
J_i=\int_{i\pi}^{(i+1)\pi}\frac{\sin^n(s)}{s^{n-2}}ds\mbox{.}
$$

In order to prove that $Q_2$ is negative, it is sufficient to show that the sum $J_i+J_{i+1}$ is positive when $i$ is even. Observe that
$$
J_i+J_{i+1}=\int_{i\pi}^{(i+1)\pi}\frac{\sin^n(s)}{s^{n-2}}-\frac{\sin^n(s)}{(s+\pi)^{n-2}}ds\mbox{.}
$$

If $i$ is even and $s$ belongs to $]i\pi,(i+1)\pi[$, then $\sin(s)$ is positive and, as an immediate consequence, $J_i+J_{i+1}$ is positive as well.
\end{proof}

\section{Local maxima away from the center of the hypercube}\label{LEH.sec.3}

Consider an integer $n$ satisfying $2\leq{n}\leq{d}$ and recall that $]0,+\infty[^n$ denotes the subset of $[0,+\infty[^d$ made up of the points whose first $n$ coordinates are positive and whose last $d-n$ coordinates are equal to $0$. In this section and the next, the local extremality of $V$ is investigated at the point $a$ of $\mathbb{S}^{d-1}\cap]0,+\infty[^n$ whose first $n$ coordinates coincide. The change of variables
\begin{equation}\label{LEH.sec.3.eq.1}
z=\frac{n}{2}-t\sqrt{n}
\end{equation}
is used throughout both section. 

The local extremality of $V$ when $t$ is large (and $z$ small) will be studied via Theorem \ref{LEH.sec.2.thm.1} as in Section \ref{LEH.sec.2}, except that (\ref{LEH.sec.2.thm.1.eq.0.1}) and (\ref{LEH.sec.2.thm.1.eq.0.2}) will be expressed as discrete sums instead of improper integrals.

\begin{lem}\label{LEH.sec.3.lem.1}
If $n$ is greater than, or equal to $4$, then at the point $a$ of $]0,+\infty[^n$ whose first $n$ coordinates are equal to $1/\sqrt{n}$,
$$
\frac{\partial^2}{\partial{a_1^2}}\frac{V}{\|a\|}-\sqrt{n}\frac{\partial}{\partial{a_1}}\frac{V}{\|a\|}-\frac{\partial^2}{\partial{a_1}\partial{a_2}}\frac{V}{\|a\|}=\sum_{i=0}^{\lfloor{z}\rfloor}\frac{(-1)^i\sqrt{n}}{(n-3)!}{n\choose{i}}(z-i)^{n-3}p_{i,n}(z)\mbox{.}
$$
\end{lem}
\begin{proof}
Consider a vertex $v$ of $[0,1]^d$. At any point $a$ in $]0,+\infty[^n$,
\begin{equation}\label{LEH.sec.3.thm.1.eq.0}
\frac{\partial}{\partial{a_1}}\frac{(b-a\mathord{\cdot}v)^{n-1}}{\pi(a)}=\frac{(b-a\mathord{\cdot}v)^{n-2}}{\pi(a)}\Biggl[(n-1)\!\left(\frac{1}{2}-v_1\right)\!-\frac{b-a\mathord{\cdot}v}{a_1}\Biggr]\!\mbox{.}
\end{equation}

Differentiating again yields
\begin{multline}\label{LEH.sec.3.thm.1.eq.1}
\displaystyle\frac{\partial^2}{\partial{a_1^2}}\frac{(b-a\mathord{\cdot}v)^{n-1}}{\pi(a)}=\frac{(b-a\mathord{\cdot}v)^{n-3}}{\pi(a)}\Biggl[\frac{(n-1)(n-2)}{4}\\
\hfill-2(n-1)\!\left(\frac{1}{2}-v_1\right)\!\frac{b-a\mathord{\cdot}v}{a_1}+2\frac{(b-a\mathord{\cdot}v)^2}{a_1^2}\Biggr]
\end{multline}
and
\begin{multline}\label{LEH.sec.3.thm.1.eq.2}
\displaystyle\frac{\partial^2}{\partial{a_1}\partial{a_2}}\frac{(b-a\mathord{\cdot}v)^{n-1}}{\pi(a)}=\frac{(b-a\mathord{\cdot}v)^{n-3}}{\pi(a)}\Biggl[(n-1)(n-2)\!\left(\frac{1}{2}-v_1\right)\!\!\left(\frac{1}{2}-v_2\right)\\
\hfill-(n-1)\!\left(\frac{1}{2}-v_2\right)\!\frac{b-a\mathord{\cdot}v}{a_1}-(n-1)\!\left(\frac{1}{2}-v_1\right)\!\frac{b-a\mathord{\cdot}v}{a_2}+\frac{(b-a\mathord{\cdot}v)^2}{a_1a_2}\Biggr]\!\mbox{.}
\end{multline}

Now denote by $\mathcal{L}_i$ the set of the vertices $v$ of the hypercube $[0,1]^n$ whose coordinates sum to some integer $i$. Recall that
$$
|\mathcal{L}_i|={n\choose{i}}\!\mbox{.}
$$

Moreover, exactly
$$
n-1\choose{i-1}
$$
vertices $v$ in $\mathcal{L}_i$ satisfy $v_1=1$. In particular,
$$
\begin{array}{rcl}
\displaystyle\sum_{v\in\mathcal{L}_i}\!\left(\frac{1}{2}-v_1\right)\! & \!\!\!\!=\!\!\!\! & \displaystyle\frac{1}{2}{n\choose{i}}-{n-1\choose{i-1}}\!\mbox{,}\\[\bigskipamount]
& \!\!\!\!=\!\!\!\! & \displaystyle{n\choose{i}}\!\left(\frac{1}{2}-\frac{i}{n}\right)\!\!\mbox{.}\\
\end{array}
$$

Assume that $a$ is the point of $]0,+\infty[^n$ whose first $n$ coordinates are equal to $1/\sqrt{n}$. In that case, for any vertex $v$ in $\mathcal{L}_i$, (\ref{LEH.sec.3.eq.1}) yields
$$
b-a\mathord{\cdot}v=\frac{z-i}{\sqrt{n}}\mbox{.}
$$

Therefore, $b-a\mathord{\cdot}v$ only depends on $i$ and, according to (\ref{LEH.sec.3.thm.1.eq.0}),
\begin{equation}\label{LEH.sec.3.thm.1.eq.2.5}
\sum_{v\in\mathcal{L}_i}\frac{\partial}{\partial{a_1}}\frac{(b-a\mathord{\cdot}v)^{n-1}}{\pi(a)}=n{n\choose{i}}(z-i)^{n-2}\Biggl[(n-1)\!\left(\frac{1}{2}-\frac{i}{n}\right)\!-(z-i)\Biggr]\!\mbox{.}
\end{equation}

Similarly, according to (\ref{LEH.sec.3.thm.1.eq.1}),
\begin{multline}\label{LEH.sec.3.thm.1.eq.3}
\sum_{v\in\mathcal{L}_i}\frac{\partial^2}{\partial{a_1^2}}\frac{(b-a\mathord{\cdot}v)^{n-1}}{\pi(a)}=n\sqrt{n}{n\choose{i}}(z-i)^{n-3}\Biggl[\frac{(n-1)(n-2)}{4}\\
\hfill-2(n-1)\!\left(\frac{1}{2}-\frac{i}{n}\right)\!(z-i)+2(z-i)^2\Biggr]\!\mbox{.}
\end{multline}

Observe that $\mathcal{L}_i$ contains exactly
$$
n-2\choose{i}
$$
vertices whose first two coordinates are both equal to $0$,
$$
2{n-2\choose{i-1}}
$$
vertices whose first two coordinates are different, and
$$
n-2\choose{i-2}
$$
vertices whose first two coordinates are both equal to $1$. In particular,
$$
\begin{array}{rcl}
\displaystyle\sum_{v\in\mathcal{L}_i}\!\left(\frac{1}{2}-v_1\right)\!\!\left(\frac{1}{2}-v_2\right)\! & \!\!\!\!=\!\!\!\! & \displaystyle\frac{1}{4}\!\left[{n-2\choose{i}}+{n-2\choose{i-2}}-2{n-2\choose{i-1}}\right]\!\!\mbox{,}\\[\bigskipamount]
& \!\!\!\!=\!\!\!\! & \displaystyle{n\choose{i}}\!\left(\frac{1}{4}-\frac{i(n-i)}{n(n-1)}\right)\!\!\mbox{.}\\[\bigskipamount]
\end{array}
$$

Hence, it follows from (\ref{LEH.sec.3.thm.1.eq.2}) that
\begin{multline}\label{LEH.sec.3.thm.1.eq.4}
\sum_{v\in\mathcal{L}_i}\frac{\partial^2}{\partial{a_1}\partial{a_2}}\frac{(b-a\mathord{\cdot}v)^{n-1}}{\pi(a)}=n\sqrt{n}{n\choose{i}}(z-i)^{n-3}\Biggl[\frac{(n-1)(n-2)}{4}\\
\hfill-\frac{i(n-i)(n-2)}{n}-2(n-1)\!\left(\frac{1}{2}-\frac{i}{n}\right)\!(z-i)+(z-i)^2\Biggr]\!\mbox{.}
\end{multline}

Since the first $n$ coordinates of $a$ are equal, a vertex $v$ of $[0,1]^n$ satisfies $a\mathord{\cdot}v\leq{b}$ if and only if its coordinates sum to at most $z$. Hence, by Lemma \ref{LEH.sec.1.lem.2},
\begin{multline*}
\frac{\partial^2}{\partial{a_1^2}}\frac{V}{\|a\|}-\sqrt{n}\frac{\partial}{\partial{a_1}}\frac{V}{\|a\|}-\frac{\partial^2}{\partial{a_1}\partial{a_2}}\frac{V}{\|a\|}=\sum_{i=0}^{\lfloor{z}\rfloor}\frac{(-1)^i}{(n-1)!}\sum_{v\in\mathcal{L}_i}\Biggl[\frac{\partial^2}{\partial{a_1^2}}\frac{(b-a\mathord{\cdot}v)^{n-1}}{\pi(a)}\\
\hfill-\sqrt{n}\frac{\partial}{\partial{a_1}}\frac{(b-a\mathord{\cdot}v)^{n-1}}{\pi(a)}-\frac{\partial^2}{\partial{a_1}\partial{a_2}}\frac{(b-a\mathord{\cdot}v)^{n-1}}{\pi(a)}\Biggr]\!\mbox{.}
\end{multline*}

Combining this with (\ref{LEH.sec.0.thm.1.eq.-1}), (\ref{LEH.sec.3.thm.1.eq.2.5}), (\ref{LEH.sec.3.thm.1.eq.3}), and (\ref{LEH.sec.3.thm.1.eq.4}) completes the proof.
\end{proof}

Theorem \ref{LEH.sec.0.thm.1} can now be proven.

\begin{proof}[Proof of Theorem \ref{LEH.sec.0.thm.1}]
Assume that $n$ is equal to $d$. The theorem is obtained as a consequence of Theorem~\ref{LEH.sec.2.thm.1} and Lemma \ref{LEH.sec.3.lem.1}. Indeed, according to the latter, (\ref{LEH.sec.0.thm.1.eq.0}) and (\ref{LEH.sec.2.thm.1.eq.0.1}) are multiples of each other by a positive number.
\end{proof}

\begin{rem}
Note that a statement equivalent to that of Theorem \ref{LEH.sec.0.thm.1} can be obtained by replacing (\ref{LEH.sec.0.thm.1.eq.0}) with the right-hand side of (\ref{LEH.sec.2.lem.1.eq.1}).
\end{rem}


In order to determine the signs of weighted alternating sum of binomial coefficients such as (\ref{LEH.sec.0.thm.1.eq.0}), the following technical statement will be used.

\begin{prop}\label{LEH.sec.3.prop.1}
Consider a non-negative integer $l$ such that $l<z$ and, for each integer~$i$ satisfying $l\leq{i}\leq{z}$, a positive number $f_i(z)$. Assume that $f_i(z)/f_{i+1}(z)$ is monotonically increasing with $i$. In this case, if
\begin{equation}\label{LEH.sec.3.prop.1.eq.0}
z<l+\frac{1}{1-\!\left(\displaystyle\frac{l+1}{n-l}\frac{f_l(z)}{f_{l+1}(z)}\right)^{1/(n-3)}}\mbox{,}
\end{equation}
then the sum
\begin{equation}\label{LEH.sec.3.prop.1.eq.1}
\sum_{i=l}^{\lfloor{z}\rfloor}(-1)^i{n\choose{i}}(z-i)^{n-3}f_i(z)
\end{equation}
is positive when $l$ is even and negative when $l$ is odd.
\end{prop}
\begin{proof}
First observe that, when $l$ is equal to $\lfloor{z}\rfloor$, the result is immediate. Therefore, it is assumed in this proof that $l$ is less than $\lfloor{z}\rfloor$. In that case, denote by $m$ the largest integer less than or equal to $z$, whose parity is different from the parity of $l$. It is sufficient to show that
\begin{equation}\label{LEH.sec.3.prop.1.eq.2}
\sum_{i=l}^m(-1)^i{n\choose{i}}(z-i)^{n-3}f_i(z)
\end{equation}
is positive when $l$ is even and negative when $l$ is odd. Indeed, that sum only possibly misses one term of (\ref{LEH.sec.3.prop.1.eq.1}), but that missing term (when there is one) necessarily has the desired sign. Now observe that there is an even number of terms in (\ref{LEH.sec.3.prop.1.eq.2}), whose signs alternate. Hence, it suffices to prove that
\begin{equation}\label{LEH.sec.3.prop.1.eq.3}
{n\choose{i}}(z-i)^{n-3}f_i(z)-{n\choose{i+1}}(z-i-1)^{n-3}f_{i+1}(z)
\end{equation}
is positive when $l\leq{i}<m$. One can see that (\ref{LEH.sec.3.prop.1.eq.3}) is positive if and only if
$$
\frac{i+1}{d-i}\frac{f_i(z)}{f_{i+1}(z)}>\!\left(1-\frac{1}{z-i}\right)^{n-3}
$$
which in turn is equivalent to
$$
z<i+\frac{1}{1-\!\left(\displaystyle\frac{i+1}{n-i}\frac{f_i(z)}{f_{i+1}(z)}\right)^{1/(n-3)}}\mbox{.}
$$

As $f_i(z)/f_{i+1}(z)$ is monotonically increasing with $i$, so is the right-hand side of this inequality and it is therefore implied by (\ref{LEH.sec.3.prop.1.eq.0}), as desired.
\end{proof}

The following can be proven using Proposition \ref{LEH.sec.3.prop.1}.

\begin{lem}\label{LEH.sec.3.lem.2}
Assume that $4\leq{n}\leq{d}$. If in addition,
\begin{equation}\label{LEH.sec.3.lem.2.eq.1}
0<z<\min\!\left\{\frac{n-1}{4},\frac{n^{1/(n-3)}}{n^{1/(n-3)}-1}\right\}\!\mbox{,}
\end{equation}
then, at the point $a$ of $]0,+\infty[^n$ whose first $n$ coordinates are equal to $1/\sqrt{n}$,
\begin{equation}\label{LEH.sec.3.lem.2.eq.1.5}
\sum_{i=0}^{\lfloor{z}\rfloor}(-1)^i{n\choose{i}}(z-i)^{n-3}p_{i,n}(z)<0\mbox{.}
\end{equation}
\end{lem}
\begin{proof}
Assume that $z$ satisfies (\ref{LEH.sec.3.lem.2.eq.1}) and note that, when $n=4$, this implies
$$
0<z<\frac{3}{4}\mbox{.}
$$

As a consequence,
$$
\sum_{i=0}^{\lfloor{z}\rfloor}(-1)^i{n\choose{i}}(z-i)^{n-3}p_{i,n}(z)=zp_{0,4}(z)\mbox{.}
$$

However, in turn,
$$
p_{0,4}(z)=\frac{4}{3}z\!\left(z-\frac{3}{4}\right)\!\!\mbox{,}
$$
and the desired result immediately follows. Now assume that $n$ is at least $5$ and observe that $p_{i,n}(z)$ can be rewritten as
\begin{equation}\label{LEH.sec.3.lem.2.eq.2}
p_{i,n}(z)=\frac{-f_i(z)+g_i(z)}{(n-1)(n-2)}
\end{equation}
where
$$
f_i(z)=2nz\!\left(\frac{n-1}{4}-z\right)\!
$$
and
$$
g_i(z)=(3n+1)\!\left(\frac{n}{2}-\frac{3n}{3n+1}-z\right)\!i+3i^2\mbox{.}
$$

Further note that $f_i(z)$ is positive when $0\leq{i}\leq{z}$ because $z$ is less than $(n-1)/4$. As in addition, $n$ is at least $4$, $g_i(z)$ is positive when $1\leq{i}\leq{z}$. Moreover, since $f_i(z)$ does not depend on $i$, it is immediate that the ratio $f_i(z)/f_{i+1}(z)$ is monotonically increasing with $i$.

On the other hand, $g_i(z)/g_{i+1}(z)$ can be written in the form
$$
\left(1-\frac{1}{i+1}\right)\!\!\left(1-\frac{1}{\alpha+i+1}\right)
$$
where
$$
\alpha=\frac{(3n+1)}{3}\!\left(\frac{n}{2}-\frac{3n}{3n+1}-z\right)\!\!\mbox{.}
$$

Again, $\alpha$ is positive because $z$ is less than $(n-1)/4$. As a consequence, the ratio $g_i(z)/g_{i+1}(z)$ is also monotonically increasing with $i$.

Now observe that $g_0(z)$ is equal to $0$. Hence, by (\ref{LEH.sec.3.lem.2.eq.2}),
\begin{multline*}
\sum_{i=0}^{\lfloor{z}\rfloor}(-1)^i{n\choose{i}}(z-i)^{n-3}p_{i,n}(z)=-\sum_{i=0}^{\lfloor{z}\rfloor}(-1)^i{n\choose{i}}\frac{f_i(z)(z-i)^{n-3}}{(n-1)(n-2)}\\
+\sum_{i=1}^{\lfloor{z}\rfloor}(-1)^i{n\choose{i}}\frac{g_i(z)(z-i)^{n-3}}{(n-1)(n-2)}\mbox{.}
\end{multline*}

According to Proposition \ref{LEH.sec.3.prop.1}, this is negative when
$$
0<z<\min\!\left\{\frac{n^{1/(n-3)}}{n^{1/(n-3)}-1},1+\frac{1}{1-q^{1/(n-3)}}\right\}\!
$$
where
$$
q=\frac{2}{n-1}\frac{g_1(z)}{g_2(z)}\mbox{.}
$$

In order to complete the proof, it suffices to show that
$$
\frac{1}{1-q^{1/(n-3)}}\geq\frac{n^{1/(n-3)}}{n^{1/(n-3)}-1}
$$
or, equivalently, that $q\geq1/n$. Observe that
$$
\begin{array}{rcl}
q & \!\!\!\!=\!\!\!\! & \displaystyle\frac{1}{n-1}\frac{(3n+1)(n-2z)-6n+6}{(3n+1)(n-2z)-6n+12}\mbox{,}\\[\bigskipamount]
& \!\!\!\!=\!\!\!\! & \displaystyle\frac{1}{n-1}\!\left(1-\frac{6}{3n^2-5n+12-(2+6n)z}\right)\!\!\mbox{.}\\
\end{array}
$$

Therefore, $q\geq1/n$ if and only if
$$
\frac{6}{3n^2-5n+12-(2+6n)z}\leq\frac{1}{n}
$$

Note that the denominator in the left-hand side of this inequality is positive. Hence, this is equivalent to the non-negativity of
$$
3n^2-11n+12-(2+6n)z
$$

However, since $z<(n-1)/4$,
$$
3n^2-11n+12-(2+6n)z>\frac{3}{2}n^2-10n+\frac{25}{2}\mbox{.}
$$

As the right-hand side of this inequality is non-negative when $n\geq5$, this shows that $q$ is at least $1/n$, as desired.
\end{proof}

Theorem \ref{LEH.sec.0.thm.0.1} can now be proven.

\begin{proof}[Proof of Theorem \ref{LEH.sec.0.thm.0.1}]
By the symmetries of the hypercube, it suffices to show that $V$ has a strict local maximum on $\mathbb{S}^{d-1}\cap]0,+\infty[^d$ at the point $a$ whose coordinates are all equal to $1/\sqrt{d}$. Assume that $n$ is equal to $d$. In that case,
$$
z=\frac{d}{2}-t\sqrt{d}
$$
and the result follows from Theorem \ref{LEH.sec.0.thm.1} and Lemma \ref{LEH.sec.3.lem.2}.
\end{proof}

%
%

\section{The sub-diagonals of the hypercube}\label{LEH.sec.3.5}

While Theorem \ref{LEH.sec.2.thm.1} only requires the negativity of (\ref{LEH.sec.2.thm.1.eq.0.1}) to treat the diagonals of the hypercube, it further needs (\ref{LEH.sec.2.thm.1.eq.0.2}) to be negative as well in the case of lower order sub-diagonals. The following lemma gives an expression of the latter quantity. As in the previous section, $n$ is an integer such that $4\leq{n}\leq{d}$ and $z$ is related to $t$ via the change of variables (\ref{LEH.sec.3.eq.1}).

\begin{lem}\label{LEH.sec.3.5.lem.1}
Assume that $n$ is less than $d$. In this case, at the point of $]0,+\infty[^n$ whose first $n$ coordinates are $1/\sqrt{n}$,
$$
\frac{\partial^2}{\partial{a_d^2}}\frac{V}{\|a\|}=\sum_{i=0}^{\lfloor{z}\rfloor}\frac{(-1)^in\sqrt{n}}{12(n-3)!}{n\choose{i}}(z-i)^{n-3}\mbox{.}
$$
\end{lem}
\begin{proof}
According to Lemma \ref{LEH.sec.1.lem.3},
\begin{equation}\label{LEH.sec.3.5.lem.1.eq.1}
\frac{\partial^2}{\partial{a_j^2}}\frac{V}{\|a\|}=\sum\frac{(-1)^{\sigma(v)}(b-a\mathord{\cdot}v)^{n-3}}{12(n-3)!\pi(a)}
\end{equation}
where the sum is over the vertices $v$ of $[0,1]^n$ satisfying $a\mathord{\cdot}v\leq{b}$. At the point $a$ of $]0,+\infty[^n$ whose first $n$ coordinates are equal to $1/\sqrt{n}$, these vertices are precisely the ones whose sum of coordinates is at most $z$. Recall that, for any integer $i$ such that $0\leq{i}\leq{z}$, the hypercube $[0,1]^n$ has exactly
$$
{n\choose{i}}
$$
vertices whose coordinates sum to $i$. As, for any such vertex $v$,
$$
\frac{(-1)^{\sigma(v)}(b-a\mathord{\cdot}v)^{n-3}}{12(n-3)!\pi(a)}=n\sqrt{n}\frac{(-1)^i(z-i)^{n-3}}{12(n-3)!}\mbox{,}
$$
the desired expression follows from (\ref{LEH.sec.3.5.lem.1.eq.1}).
\end{proof}

Theorem \ref{LEH.sec.0.thm.2} can now be proven.

\begin{proof}[Proof of Theorem \ref{LEH.sec.0.thm.2}]
Assume that $n$ is less than $d$. According to Lemma \ref{LEH.sec.3.lem.1} (\ref{LEH.sec.0.thm.2.eq.0}) and (\ref{LEH.sec.2.thm.1.eq.0.1}) have the same sign. By Lemma \ref{LEH.sec.3.5.lem.1}, (\ref{LEH.sec.0.thm.2.eq.1}) and (\ref{LEH.sec.2.thm.1.eq.0.2}) also have the same sign and the result follows from Theorem~\ref{LEH.sec.2.thm.1}.
\end{proof}

\begin{rem}
Theorem \ref{LEH.sec.0.thm.2} can be stated equivalently using the right-hand side of (\ref{LEH.sec.2.lem.1.eq.1}) instead of (\ref{LEH.sec.0.thm.2.eq.0}) and the right-hand side of (\ref{LEH.sec.2.lem.1.eq.2}) instead of (\ref{LEH.sec.0.thm.2.eq.1}).
\end{rem}

The results established so far allow to prove the local extremality of $V$ at certain points. The following theorem makes it possible to prove that, at these points, $V$ is not locally extremal, even weakly so. Its proof is similar to that of Theorem~\ref{LEH.sec.2.thm.1}, except that it relies on the necessary conditions of the Lagrange multipliers theorem instead of the sufficient conditions.

\begin{thm}\label{LEH.sec.3.5.thm.1}
Consider an integer $n$ satisfying $2\leq{n}<d$ and assume that $V$ is twice continuously differentiable at the point $a$ of $\mathbb{R}^n$ whose first $n$ coordinates are equal to $1/\sqrt{n}$. If at that point,
\begin{equation}\label{LEH.sec.3.5.thm.1.eq.0.1}
\frac{\partial^2}{\partial{a_1^2}}\frac{V}{\|a\|}-\sqrt{n}\frac{\partial}{\partial{a_1}}\frac{V}{\|a\|}-\frac{\partial^2}{\partial{a_1}\partial{a_2}}\frac{V}{\|a\|}
\end{equation}
and
\begin{equation}\label{LEH.sec.3.5.thm.1.eq.0.2}
\frac{\partial^2}{\partial{a_d^2}}\frac{V}{\|a\|}
\end{equation}
are both non-zero and have opposite signs, then $V$ does not have a local extremum (even a weak one) on $\mathbb{S}^{d-1}\cap[0,+\infty[^d$ at $a$.
\end{thm}
\begin{proof}
As in the proof of Theorem \ref{LEH.sec.2.thm.1}, denote
$$
\lambda=-\frac{\sqrt{n}}{2}\frac{\partial}{\partial{a_1}}\frac{V}{\|a\|}\mbox{.}
$$

Here and in the remainder of the proof, all partial derivatives are taken at the point $a$ of $\mathbb{R}^n$ whose first $n$ coordinates are equal to $1/\sqrt{n}$.

Further denote
$$
L_\lambda=\frac{V}{\|a\|}+\lambda\!\left(\|a\|^2-1\right)\!.
$$

With the above choice for $\lambda$, one obtains from the same argument as in the proof of Theorem \ref{LEH.sec.2.thm.1} that $a$ is a critical point of $L_\lambda$. By the necessary conditions of the Lagrange multipliers theorem (see for instance Proposition 3.1.1 from \cite{Bertsekas1999}), if $V$ has a local maximum on $\mathbb{S}^{d-1}\cap]0,+\infty[^n$ at point $a$ then, for every non-zero point $x$ in $\mathbb{R}^d$ whose first $n$ coordinates sum to $0$,
\begin{equation}\label{LEH.sec.3.5.thm.1.eq.1}
\sum_{j=1}^d\sum_{k=1}^dx_jx_k\frac{\partial^2{L_\lambda}}{\partial{a_j}\partial{a_k}}\leq0\mbox{.}
\end{equation}

However, as shown in the proof of Theorem \ref{LEH.sec.2.thm.1},
\begin{multline*}
\sum_{j=1}^d\sum_{k=1}^dx_jx_k\frac{\partial^2{L_\lambda}}{\partial{a_j}\partial{a_k}}=\!\left[\frac{\partial^2}{\partial{a_1^2}}\frac{V}{\|a\|}-\sqrt{n}\frac{\partial}{\partial{a_1}}\frac{V}{\|a\|}-\frac{\partial^2}{\partial{a_1}\partial{a_2}}\frac{V}{\|a\|}\right]\!\sum_{i=1}^nx_i^2\\
+\frac{\partial^2}{\partial{a_1}\partial{a_2}}\frac{V}{\|a\|}\!\left[\sum_{i=1}^nx_i\right]^2\!+\frac{\partial^2}{\partial{a_d^2}}\frac{V}{\|a\|}\sum_{i=n+1}^dx_i^2\mbox{.}
\end{multline*}

Hence, taking for $x$ the point of $\mathbb{R}^d$ whose first two coordinates are $1$ and $-1$ and whose all other coordinates are equal to $0$ yields
$$
\sum_{j=1}^d\sum_{k=1}^dx_jx_k\frac{\partial^2{L_\lambda}}{\partial{a_j}\partial{a_k}}=2\!\left[\frac{\partial^2}{\partial{a_1^2}}\frac{V}{\|a\|}-\sqrt{n}\frac{\partial}{\partial{a_1}}\frac{V}{\|a\|}-\frac{\partial^2}{\partial{a_1}\partial{a_2}}\frac{V}{\|a\|}\right]\!\mbox{,}
$$
and taking, for $x$ the point of $\mathbb{R}^d$ whose first $d-1$ coordinates are equal to $0$ and whose last coordinate is equal to $1$,
$$
\sum_{j=1}^d\sum_{k=1}^dx_jx_k\frac{\partial^2{L_\lambda}}{\partial{a_j}\partial{a_k}}=\frac{\partial^2}{\partial{a_d^2}}\frac{V}{\|a\|}\mbox{.}
$$

Therefore, when (\ref{LEH.sec.3.5.thm.1.eq.0.1}) and (\ref{LEH.sec.3.5.thm.1.eq.0.2}) are non-zero and have opposite signs, then $V$ does not have a local maximum on $\mathbb{S}^{d-1}\cap[0,+\infty[^d$ at $a$. By the same argument, but with the necessary conditions of the Lagrange multipliers theorem for local minima instead of maxima, one obtains that, under the same assumptions, $V$ cannot have a local minimum on $\mathbb{S}^{d-1}\cap[0,+\infty[^d$ at $a$.
\end{proof}

Theorem \ref{LEH.sec.0.thm.0.3} is now proven as a consequence of Theorem \ref{LEH.sec.3.5.thm.1}.

\begin{proof}[Proof of Theorem \ref{LEH.sec.0.thm.0.3}]
Assume that $4\leq{n}<d$ and that $t$ satisfies (\ref{LEH.sec.0.thm.0.3.eq.1}). By the symmetries of the hypercube, it suffices to show that $V$ does not have a local extremum on $\mathbb{S}^{d-1}\cap[0,+\infty[^d$ at the point $a$ of $\mathbb{R}^n$ whose first $n$ coordinates are equal to $1/\sqrt{n}$. Recall that $z$ and $t$ are linked via the change of variables (\ref{LEH.sec.3.eq.1}). In particular, (\ref{LEH.sec.0.thm.0.3.eq.1}) is equivalent to the condition on $x$ in the statement of Lemma \ref{LEH.sec.3.lem.2}. Hence, according to that lemma and to Lemma \ref{LEH.sec.3.lem.1},
$$
\frac{\partial^2}{\partial{a_1^2}}\frac{V}{\|a\|}-\sqrt{n}\frac{\partial}{\partial{a_1}}\frac{V}{\|a\|}-\frac{\partial^2}{\partial{a_1}\partial{a_2}}\frac{V}{\|a\|}
$$
is negative at the point $a$ of $\mathbb{R}^n$ whose first $n$ coordinates are equal to $1/\sqrt{n}$. Therefore, by Theorem \ref{LEH.sec.3.5.thm.1}, it suffices to show that
$$
\frac{\partial^2}{\partial{a_d^2}}\frac{V}{\|a\|}
$$
is positive at that point. Lemma \ref{LEH.sec.3.5.lem.1} yields
$$
\frac{\partial^2}{\partial{a_d^2}}\frac{V}{\|a\|}=\sum_{i=0}^{\lfloor{z}\rfloor}\frac{(-1)^in\sqrt{n}}{12(n-3)!}{n\choose{i}}(z-i)^{n-3}\mbox{,}
$$
which, according to Proposition \ref{LEH.sec.3.prop.1} is positive when
$$
0<z<\frac{n^{1/(n-3)}}{n^{1/(n-3)}-1}\mbox{,}
$$

According to (\ref{LEH.sec.3.eq.1}), this is implied by (\ref{LEH.sec.0.thm.0.3.eq.1}), as desired.
\end{proof}

\section{Low dimensional hypercubes}\label{LEH.sec.4}

As observed in the proof of Lemma \ref{LEH.sec.3.lem.2}, 
$$
\sum_{i=0}^{\lfloor{z}\rfloor}(-1)^i{d\choose{i}}(z-i)^{d-3}p_{i,d}(z)=\frac{4}{3}z^2\!\left(z-\frac{3}{4}\right)
$$
when $d$ is equal to $4$ and $0<z<1$. This made it possible to show that if
$$
\frac{\sqrt{4}}{2}-\frac{3}{4}\frac{1}{\sqrt{4}}<t<\frac{\sqrt{4}}{2}\mbox{.}
$$
then the $3$-dimensional volume of $H\cap[0,1]^4$ is strictly locally maximal when $H$ is orthogonal to a diagonal of the $4$-dimensional hypercube $[0,1]^4$. By Theorem~\ref{LEH.sec.0.thm.1}, this also immediately proves that, if
$$
\frac{\sqrt{4}}{2}-\frac{1}{\sqrt{4}}\leq{t}<\frac{\sqrt{4}}{2}-\frac{3}{4}\frac{1}{\sqrt{4}}\mbox{,}
$$
then the $3$-dimensional volume of $H\cap[0,1]^4$ is strictly locally minimal when $H$ is orthogonal to a diagonal of $[0,1]^d$. The range for $t$ when local minimality occurs can be completed. Indeed, when $d=4$ and $1\leq{z}\leq{2}$,
$$
\begin{array}{rcl}
\displaystyle\sum_{i=0}^{\lfloor{z}\rfloor}(-1)^i{d\choose{i}}(z-i)^{d-3}p_{i,d}(z) & \!\!\!\!=\!\!\!\! & \displaystyle zp_{0,4}(z)-4(z-1)p_{1,4}(z)\mbox{,}\\
& \!\!\!\!=\!\!\!\! & \displaystyle -4z^3+17z^2-24z+\frac{34}{3}\mbox{.}\\
\end{array}
$$

This polynomial admits, as its unique real root
\begin{equation}\label{LEH.sec.4.eq.1}
\rho_4^-=\frac{17+(17-12\sqrt{2})^{1/3}+(17+12\sqrt{2})^{1/3}}{12}
\end{equation}
which is about $1.71229$. Moreover, it is positive when $z<\rho_4^-$, and negative when $z>\rho_4^-$. As a consequence, by Theorem \ref{LEH.sec.0.thm.1}, if
$$
\frac{\sqrt{4}}{2}-\frac{\rho_4^-}{\sqrt{4}}<t<\frac{\sqrt{4}}{2}-\frac{3}{4}\frac{1}{\sqrt{4}}\mbox{,}
$$
then the $3$-dimensional volume of $H\cap[0,1]^4$ is strictly locally minimal when $H$ is orthogonal to a diagonal of $[0,1]^4$ and if
$$
0<t<\frac{\sqrt{4}}{2}-\frac{\rho_4^-}{\sqrt{4}}\mbox{,}
$$
then that volume becomes strictly locally maximal again when $H$ is orthogonal to a diagonal of the $4$-dimensional hypercube $[0,1]^4$. These observations carry over to (at least) the first few higher dimensions.

\begin{prop}\label{LEH.sec.4.prop.1}
Assume that $4\leq{d}\leq7$. There exist two numbers $\rho_d^-$ and $\rho_d^+$ such that $0<\rho_d^+<\rho_d^-<d/2$ and, if 
$$
t\in\!\left]0,\frac{\sqrt{d}}{2}-\frac{\rho_d^-}{\sqrt{d}}\right[\!\cup\!\left]\frac{\sqrt{d}}{2}-\frac{\rho_d^+}{\sqrt{d}},\frac{\sqrt{d}}{2}\right[\!
$$
then the $(d-1)$-dimensional volume of $H\cap[0,1]^d$ is strictly locally maximal when $H$ is orthogonal to a diagonal of $[0,1]^d$. However, if
$$
t\in\!\left]\frac{\sqrt{d}}{2}-\frac{\rho_d^-}{\sqrt{4}},\frac{\sqrt{d}}{2}-\frac{\rho_d^+}{\sqrt{4}}\right[\!
$$
then the $(d-1)$-dimensional volume of $H\cap[0,1]^d$ is strictly locally minimal when $H$ is orthogonal to a diagonal of $[0,1]^d$.
\end{prop}
\begin{proof}
The proposition has been proven above when $d$ is equal to $4$. Assume that $d\geq5$ and recall that the desired result is proven in \cite{Konig2021} when
$$
\frac{\sqrt{d}}{2}-\frac{1}{\sqrt{d}}<t<\frac{\sqrt{d}}{2}\mbox{.}
$$

Assume first that $d$ is equal to $5$. If in addition, $1\leq{z}\leq{2}$, then
$$
\sum_{i=0}^{\lfloor{z}\rfloor}(-1)^i{d\choose{i}}(z-i)^{d-3}p_{i,d}(z)=-\frac{5}{6}(4z^2-10z+7)(z-1)(z-2)
$$
and if $2<z\leq{5/2}$, then
$$
\sum_{i=0}^{\lfloor{z}\rfloor}(-1)^i{d\choose{i}}(z-i)^{d-3}p_{i,d}(z)=\frac{5}{2}(2z^2-10z+13)(z-2)(z-3)\mbox{.}
$$

Since $4z^2-10z+7$ and $2z^2-10z+13$ are both always positive, the desired result follows form Theorem \ref{LEH.sec.0.thm.1} with $\rho_5^-=2$ and $\rho_5^+=1$. Now assume that $d$ is equal to $6$. In that case, when $1\leq{z}\leq{2}$,
$$
\sum_{i=0}^{\lfloor{z}\rfloor}(-1)^i{d\choose{i}}(z-i)^{d-3}p_{i,d}(z)=-3z^5+\frac{81}{4}z^4-54z^3+72z^2-48z+\frac{63}{5}\mbox{.}
$$

This polynomial is negative when $z=1$ and positive, when $z=2$. It is also increasing in the interval $[1,2]$ (this can be easily seen from its derivative, a degree $4$ polynomial that admits $1$ and $2$ as its only real roots). Hence this polynomial has a unique root $\rho_6^+$ in the interval $]1,2[$.

Further observe that, when $2<z\leq{3}$,
\begin{multline*}
\sum_{i=0}^{\lfloor{z}\rfloor}(-1)^i{d\choose{i}}(z-i)^{d-3}p_{i,d}(z)=6z^5-\frac{147}{2}z^4+360z^3-882z^2+1080z-\frac{2637}{5}\mbox{.}
\end{multline*}

This polynomial is decreasing in the interval $]2,3[$ (which, again can be easily seen from its derivative), positive when $z=2$, and negative when $z=3$. Denoting by $\rho_6^-$ the only root of this polynomial contained in that interval, the desired result follows again from Theorem \ref{LEH.sec.0.thm.1}.

Finally, assume that $d$ is equal to $7$. Using the same kind of straightforward arguments (but using the first two derivatives instead of just the first), one easily shows that, 
there exist two numbers $z_1$ and $z_2$, the first in $]1,2[$ and the other in $]2,3[$, such that the quantity
\begin{equation}\label{LEH.sec.4.prop.1.eq.1}
\sum_{i=0}^{\lfloor{z}\rfloor}(-1)^i{d\choose{i}}(z-i)^{d-3}p_{i,d}(z)\mbox{,}
\end{equation}
thought of as a function of $z$, is strictly decreasing in $[1,z_1[$, strictly increasing in $]z_1,z_2[$, and strictly decreasing again in $]z_2,7/2]$. Moreover, (\ref{LEH.sec.4.prop.1.eq.1}) is negative when $z$ is equal to $1$, $3$, and $7/2$ and positive when $z=2$.

Therefore, there exist two numbers $\rho_7^-$ and  $\rho_7^+$, the first in $]2,3[$ and the other in $]1,2[$, such that (\ref{LEH.sec.4.prop.1.eq.1}) vanishes when $z$ is equal to either of them, is negative when $z$ belongs to $[1,\rho_7^+[\cup]\rho_7^-,7/2]$, and positive when $z$ is in $]\rho_7^+,\rho_7^-[$. Hence, the proposition follows once more from Theorem \ref{LEH.sec.0.thm.1}.
\end{proof}

\begin{rem}\label{LEH.sec.4.rem.1}
As discussed above, $\rho_4^+$ is equal to $3/4$ and $\rho_4^-$, given by (\ref{LEH.sec.4.eq.1}), is about $1.71229$. It is also explicit in the proof of Proposition \ref{LEH.sec.4.prop.1} that $\rho_5^+$ is $1$ and $\rho_5^-$ is $2$. However, $\rho_6^+$ and $\rho_6^-$ do not have exact expressions. They are about $1.39766$ and $2.46963$, respectively. Likewise, $\rho_7^+$ and $\rho_7^-$ cannot be expressed exactly, but they are about $1.77221$ and $2.9324$, respectively.
\end{rem}

A similar result can be obtained for low order sub-diagonals. This can be illustrated with order $4$ sub-diagonals of higher dimensional hypercubes. Indeed, assume that $n=4$ and $d>4$. Observe that, if $0<z<1$ then
$$
\sum_{i=0}^{\lfloor{z}\rfloor}(-1)^i{n\choose{i}}(z-i)^{n-3}=z\mbox{.}
$$

Moreover, if $1\leq{z}\leq{2}$,
\begin{equation}\label{LEH.sec.4.eq.2}
\sum_{i=0}^{\lfloor{z}\rfloor}(-1)^i{n\choose{i}}(z-i)^{n-3}=-3\!\left(z-\frac{4}{3}\right)\!\!\mbox{.}
\end{equation}

Therefore, according to Theorem \ref{LEH.sec.0.thm.2} and the discussion about $\rho_4^+$ and $\rho_4^-$ from the beginning of the section, if
$$
0<t<\frac{\sqrt{4}}{2}-\frac{\rho_4^-}{\sqrt{4}}\mbox{,}
$$
then the $(d-1)$-dimensional volume of $H\cap[0,1]^d$ is strictly locally maximal when $H$ is orthogonal to an order $4$ sub-diagonal of $[0,1]^d$. Likewise, if
$$
\frac{\sqrt{4}}{2}-\frac{4}{3}\frac{1}{\sqrt{4}}<t<\frac{\sqrt{4}}{2}-\frac{\rho_4^+}{\sqrt{4}}\mbox{,}
$$
where the $4/3$ is the left-hand side is the value of $z$ such that (\ref{LEH.sec.4.eq.2}) vanishes, then that volume is strictly locally minimal when $H$ is orthogonal to an order $4$ sub-diagonal of $[0,1]^d$. Moreover, by Theorem \ref{LEH.sec.3.5.thm.1} (where (\ref{LEH.sec.3.5.thm.1.eq.0.1}) and (\ref{LEH.sec.3.5.thm.1.eq.0.2}) are expressed using Lemmas \ref{LEH.sec.3.lem.1} and \ref {LEH.sec.3.5.lem.1}), if
$$
\frac{\sqrt{4}}{2}-\frac{\rho_4^-}{\sqrt{4}}<t<\frac{\sqrt{4}}{2}-\frac{4}{3}\frac{1}{\sqrt{4}}
$$
or
$$
\frac{\sqrt{4}}{2}-\frac{\rho_4^+}{\sqrt{4}}<t\leq\frac{\sqrt{4}}{2}
$$
then the $(d-1)$-dimensional volume of $H\cap[0,1]^d$ is not locally extremal when $H$ is orthogonal to an order $4$ sub-diagonal of $[0,1]^d$. This observation carries over to the next few sub-diagonal orders.

\begin{prop}\label{LEH.sec.4.prop.2}
Assume that $4\leq{n}\leq7$ and that $n<d$. There exists a number $\rho_n^\circ$ independent on $d$ such that $\rho_n^+<\rho_n^\circ<\rho_n^-$ and, if 
$$
t\in\!\left]0,\frac{\sqrt{n}}{2}-\frac{\rho_n^-}{\sqrt{n}}\right[\!
$$
then the $(d-1)$-dimensional volume of $H\cap[0,1]^d$ is strictly locally maximal when $H$ is orthogonal to an order $n$ sub-diagonal of $[0,1]^d$. However, if
$$
t\in\!\left]\frac{\sqrt{n}}{2}-\frac{\rho_n^\circ}{\sqrt{n}},\frac{\sqrt{n}}{2}-\frac{\rho_n^+}{\sqrt{n}}\right[\!
$$
then the $(d-1)$-dimensional volume of $H\cap[0,1]^d$ is strictly locally minimal when $H$ is orthogonal to an order $n$ sub-diagonal of $[0,1]^d$. Finally, if
$$
t\in\!\left]\frac{\sqrt{n}}{2}-\frac{\rho_n^-}{\sqrt{n}},\frac{\sqrt{n}}{2}-\frac{\rho_n^\circ}{\sqrt{n}}\right[\!\cup\!\left]\frac{\sqrt{n}}{2}-\frac{\rho_n^+}{\sqrt{n}},\frac{\sqrt{n}}{2}\right[\!
$$
then that volume is not locally minimal or maximal (even weakly so) when $H$ is orthogonal to an order $n$ sub-diagonal of $[0,1]^d$. 
\end{prop}
\begin{proof}
The proposition is already proven above when $n$ is equal to $4$ and it can be assumed that $n$ is at least $5$. Following the argument used for the $4$\nobreakdash-dimensional case, it suffices to show that
\begin{equation}\label{LEH.sec.4.prop.2.eq.1}
\sum_{i=0}^{\lfloor{z}\rfloor}(-1)^i{n\choose{i}}(z-i)^{n-3}
\end{equation}
is positive when $z$ belongs to $]0,\rho_n^\circ[$ and negative when $z$ is in $]\rho_n^\circ,n/2[$, where $\rho_n^\circ$ is a number satisfying $\rho_n^+<\rho_n^\circ<\rho_n^-$. First observe that (\ref{LEH.sec.4.prop.2.eq.1}) is equal to $z^{n-3}$ when $0<z<1$, and, therefore is always positive in that case.

If $n$ is equal to $5$, (\ref{LEH.sec.4.prop.2.eq.1}) is equal to $-4z^2+10z-5$ when $1\leq{z}\leq2$ and to $6z^2-30z+35$ when $2<z\leq{5/2}$. Hence, it is positive when $z$ belongs to the interval $]0,\rho_5^\circ[$ and negative when $z$ is in the interval $]\rho_5^\circ,5/2[$, where
$$
\rho_5^\circ=\frac{5+\sqrt{5}}{4}\mbox{,}
$$
which is about $1.80902$. By Remark \ref{LEH.sec.4.rem.1}, $\rho_5^+<\rho_5^\circ<\rho_5^-$, as desired. Now assume that $n$ is equal to $6$. In that case, (\ref{LEH.sec.4.prop.2.eq.1}) is equal to
$$
-5z^3+18z^2-18z+6
$$
when $1\leq{z}\leq{2}$ and to
$$
10z^3-72z^2+162z-114
$$
when $2<z\leq{3}$. The former polynomial expression is positive when $z$ belongs to $[1,2]$, and the latter admits a single root $\rho_6^\circ$ in $[2,3]$ of about $2.2407$. Moreover, the latter expression is positive when $z$ belongs to $[2,\rho_6^\circ[$ and negative when $z$ belongs to $]\rho_6^\circ,3]$. As in addition, $\rho_6^\circ$ is strictly between $\rho_6^+$ and $\rho_6^-$ (see Remark~\ref{LEH.sec.4.rem.1}), the proposition holds in that case.

Finally assume that $n$ is equal to $7$. If $1\leq{z}\leq{2}$, then (\ref{LEH.sec.4.prop.2.eq.1}) is equal to
$$
-6z^4+28z^3-42z^2+28z-7\mbox{.}
$$

If $2<z\leq3$, then (\ref{LEH.sec.4.prop.2.eq.1}) is equal to
$$
15z^4-140z^3+462z^2-644z+329\mbox{.}
$$

If, however $3<z\leq7/2$, then (\ref{LEH.sec.4.prop.2.eq.1}) is equal to
$$
-20z^4+280z^3-1428z^2+3136z-2506\mbox{.}
$$

Note that the roots and positivity of these polynomial expressions can still be obtained exactly. In particular, it can be easily shown that (\ref{LEH.sec.4.prop.2.eq.1}) is positive when $z$ belongs to $[1,\rho_7^\circ[$ and negative when $z$ is in $]\rho_7^\circ,7/2]$, where $\rho_7^\circ$, a root of the second polynomial expression, is about $2.69068$. Since, by Remark \ref{LEH.sec.4.rem.1}, $\rho_7^+$ is about $1.77221$ and $\rho_7^-$ about $2.9324$, this proves the desired property.
\end{proof}

\begin{rem}\label{LEH.sec.4.rem.2}
Recall that $\rho_4^\circ$ is equal to $4/3$ and $\rho_5^\circ$ to
$$
\frac{5+\sqrt{5}}{4}\mbox{,}
$$
which is about $1.80902$. In fact, $\rho_6^\circ$ and $\rho_7^\circ$ (that are about $2.2407$ and $2.69068$, respectively) can also be expressed exactly. However, while $\rho_6^\circ$ can be expressed in a reasonably simple form as
$$
\rho_6^\circ=\frac{12}{5}-\frac{3}{5}\cos\!\left(\frac{1}{3}\arctan\!\left(\frac{5\sqrt{11}}{7}\right)\!\right)\!+\frac{3\sqrt{3}}{5}\sin\!\left(\frac{1}{3}\arctan\!\left(\frac{5\sqrt{11}}{7}\right)\!\right)\!\!\mbox{,}
$$
the exact expression of $\rho_7^\circ$ is too long to be written here.
\end{rem}

Interestingly, Propositions \ref{LEH.sec.4.prop.1} and \ref{LEH.sec.4.prop.2} still hold for dimensions and sub-diagonal orders much larger than $7$: they have been verified up to dimension and sub-diagonal order $300$ using symbolic computation, and can be expected to remain true beyond that. Approximate values of the corresponding $\rho_d^-$, $\rho_d^+$, and $\rho_d^0$ are reported in the following table when $8\leq{d}\leq35$.

\begin{table}[b]
\begin{tabular}{c|ccc||c|ccc}
$d$ & $\rho_d^-$ & $\rho_d^\circ$ & $\rho_d^+$ & $d$ & $\rho_d^-$ & $\rho_d^\circ$ & $\rho_d^+$\\
\hline
$8$ & $3.38859$ & $3.14086$ & $2.13730$ & $22$ & $9.99303$ & $9.62096$ & $7.86693$\\
$9$ & $3.85428$ & $3.59394$ & $2.52065$ & $23$ & $10.4705$ & $10.0911$ & $8.29529$\\
$10$ & $4.31894$ & $4.04931$ & $2.90984$ & $24$ & $10.9485$ & $10.5619$ & $8.72522$\\
$11$ & $4.78630$ & $4.50661$ & $3.30377$ & $25$ & $11.4269$ & $11.0332$ & $9.15661$\\
$12$ & $5.25481$ & $4.96566$ & $3.70286$ & $26$ & $11.9057$ & $11.5051$ & $9.58938$\\
$13$ & $5.72466$ & $5.42625$ & $4.10615$ & $27$ & $12.3849$ & $11.9775$ & $10.0234$\\
$14$ & $6.19563$ & $5.88821$ & $4.51316$ & $28$ & $12.8646$ & $12.4504$ & $10.4587$\\
$15$ & $6.66760$ & $6.35143$ & $4.92352$ & $29$ & $13.3445$ & $12.9237$ & $10.8952$\\
$16$ & $7.14049$ & $6.81577$ & $5.33689$ & $30$ & $13.8249$ & $13.3975$ & $11.3328$\\
$17$ & $7.61421$ & $7.28115$ & $5.75298$ & $31$ & $14.3055$ & $13.8717$ & $11.7714$\\
$18$ & $8.08869$ & $7.74749$ & $6.17156$ & $32$ & $14.7864$ & $14.3464$ & $12.2110$\\
$19$ & $8.56386$ & $8.21469$ & $6.59241$ & $33$ & $15.2677$ & $14.8214$ & $12.6515$\\
$20$ & $9.03967$ & $8.68271$ & $7.01536$ & $34$ & $15.7492$ & $15.2967$ & $13.0930$\\
$21$ & $9.51608$ & $9.15149$ & $7.44025$ & $35$ & $16.2310$ & $15.7725$ & $13.5353$\\
\end{tabular}
\end{table}

\noindent{\bf Acknowledgement.} The author would like to thank Hermann K{\"o}nig for sharing his notes on the problem that helped confirming the results obtained here, Artem Zvavitch for enlightening exchanges about hypercube sections, and Antoine Deza for helping with the preliminary versions of this article.

\bibliography{LocalExtremaHypercube}

\providecommand{\MR}{\relax\ifhmode\unskip\space\fi MR }
\providecommand{\MRhref}[2]{%
  \href{http://www.ams.org/mathscinet-getitem?mr=#1}{#2}
}
\providecommand{\href}[2]{#2}
\begin{thebibliography}{10}

\bibitem{Ball1986}
Keith Ball, \textsl{Cube slicing in ${R}^n$}, Proceedings of the American
  Mathematical Society \textbf{97} (1986), 465--473.

\bibitem{BarrowSmith1979}
David~L. Barrow and Philip~W. Smith, \textsl{Spline notation applied to a
  volume problem}, American Mathematical Monthly \textbf{86} (1979), 50--51.

\bibitem{BartheKoldobsky2003}
Franck Barthe and Alexander Koldobsky, \textsl{Extremal slabs in the cube and
  the {Laplace} transform}, Advances in Mathematics \textbf{174} (2003),
  89--114.

\bibitem{Berger2010}
Marcel Berger, \textsl{Geometry revealed, a {Jacob's} ladder to modern higher
  geometry}, Springer, 2010.

\bibitem{Bertsekas1999}
Dimitri~P. Bertsekas, \textsl{Nonlinear programming}, Athena Scientific,
  Belmont, Massachussetts, 1999.

\bibitem{FrankRiede2012}
Rolfdieter Frank and Harald Riede, \textsl{Hyperplane sections of the
  $n$-dimensional cube}, American Mathematical Monthly \textbf{119} (2012),
  868--872.

\bibitem{Hensley1979}
Douglas Hensley, \textsl{Slicing the cube in ${R}^n$ and probability},
  Proceedings of the American Mathematical Society \textbf{73} (1979), 95--100.

\bibitem{Koldobsky2005}
Alexander Koldobsky, \textsl{Fourier analysis in convex geometry}, Mathematical
  Surveys and Monographs, vol. 116, American Mathematical Society, 2005.

\bibitem{Konig2021}
Hermann K{\"o}nig, \textsl{Non-central sections of the simplex, the
  cross-polytope and the cube}, Advances in Mathematics \textbf{376} (2021),
  107458.

\bibitem{Konig2021b}
Hermann K{\"o}nig, personal communication, April 2021.

\bibitem{KonigKoldobsky2011}
Hermann K{\"o}nig and Alexander Koldobsky, \textsl{Volumes of low-dimensional
  slabs and sections in the cube}, Advances in Applied Mathematics \textbf{47}
  (2011), 894--907.

\bibitem{LiuTkocz2020}
Ruoyuan Liu and Tomasz Tkocz, \textsl{A note on extremal noncentral sections of
  the cross-polytope}, Advances in Applied Mathematics \textbf{118} (2020),
  102031.

\bibitem{MeyerPajor1988}
Mathieu Meyer and Alain Pajor, \textsl{Sections of the unit ball of $l_p^n$},
  Journal of Functional Analysis \textbf{80} (1988), 109--123.

\bibitem{MoodyStoneZachZvavitch2013}
James Moody, Corey Stone, David Zach, and Artem Zvavitch, \textsl{A remark on
  extremal non-central sections of the unit cube}, Asymptotic geometric
  analysis, Fields Institute Communications, vol.~68, Springer, 2013,
  pp.~211--228.

\bibitem{Polya1913}
Georg P{\'o}lya, \textsl{Berechnung eines bestimmten {Integrals}},
  Mathematische Annalen \textbf{74} (1913), 204--212.

\bibitem{Pournin2021}
Lionel Pournin, \textsl{Shallow sections of the hypercube}, Israel Journal of
  Mathematics \textbf{255} (2023), 685--704.

\bibitem{Webb1996}
Simon Webb, \textsl{Central slices of the regular simplex}, Geometriae Dedicata
  \textbf{61} (1996), 19--28.

\bibitem{Zong2006}
Chuanming Zong, \textsl{The cube, a window to convex and discrete geometry},
  Cambridge University Press, 2006.

\end{thebibliography}
\bibliographystyle{ijmart}

\end{document}